%% file: ex_article.tex
\documentclass[review,hidelinks,onefignum,onetabnum]{siamart220329}

\nolinenumbers
\input{ex_shared}
\usepackage{star}
\def\##1\#{\begin{align}#1\end{align}}
\def\$#1\${\begin{align*}#1\end{align*}}
\usepackage{mathtools}
\newcommand{\Exp}{{\rm{Exp}}}
\newcommand{\Log}{{\rm{Log}}}
\newcommand{\Inj}{{\rm{Inj}}}
\newcommand{\inj}{{\rm{inj}}}
\newcommand{\Pre}{{\rm{Pre}}}
\newcommand{\Corr}{{\rm{Corr}}}
\newcommand{\Stab}{{\rm{Stab}}}
\newcommand{\St}{{\rm{St}}}
\newcommand{\rank}{{\rm{rank}}}

\renewcommand{\dim}{{\rm{dim}}}
\renewcommand{\zero}{{\textbf{0}}}
\newcommand{\Diag}{{\rm Diag}}
\newcommand{\orb}{{\rm orb}}
\newcommand{\qt}{{\rm qt}}
\renewcommand{\diag}{{\rm diag}}
\newcommand{\grad}{{\rm grad}}
\renewcommand{\skew}{{\rm skew}}
\newcommand{\Hess}{{\rm Hess}}
\usepackage{enumitem}
\usepackage{graphicx}
\newcommand{\Frechet}{{Fr{$\rm \acute{e}$}chet$\ $}}
\ifpdf
\hypersetup{
  pdftitle={Quotient geometry of correlation matrices},
  pdfauthor={Hengchao Chen}
}
\fi

\begin{document}

\maketitle

\begin{abstract}
    This paper studies the quotient geometry of bounded or fixed-rank correlation matrices. We establish a bijection between the set of bounded-rank correlation matrices and a quotient set of a spherical product manifold by an orthogonal group. We show that it forms an orbit space, whose stratification is determined by the rank of the matrices, and the principal stratum has a compatible Riemannian quotient manifold structure. We show that any minimizing geodesic in the orbit space has constant rank on the interior of the segment. We also develop efficient Riemannian optimization algorithms for computing the distance and weighted the \Frechet mean in the orbit space. Moreover, we examine geometric properties of the quotient manifold, including horizontal and vertical spaces, Riemannian metric, injectivity radius, exponential and logarithmic map, curvature, gradient and Hessian. Finally, we apply our approach to a functional connectivity study using the Autism Brain Imaging Data Exchange.
\end{abstract}



\section{Introduction}

A correlation matrix is a positive semi-definite (PSD) matrix with unit diagonal elements. The set of all correlation matrices is called the elliptope. The elliptope arises in combinatorial optimization~\cite{tropp2018simplicial} and in many areas involving time series data, such as brain connectivity studies based on functional MRI~\cite{varoquaux2010detection},  electroencephalogram (EEG) \cite{sanei2013eeg}, or magnetoencephalogram (MEG)~\cite{da2013eeg,mellinger2007meg},  finance~\cite{archakov2021new,rebonato2000most,laloux1999noise}, and psychology~\cite{steiger1980tests}. A fundamental question is to measure the distance between two correlation matrices. Using the Euclidean distance for PSD or correlation matrices is less favorable as geodesics leave the space in finite time. To solve this problem, many non-Euclidean metrics have been proposed in the last decades.

Most of these metrics are proposed for symmetric positive definite (SPD) matrices. Examples include affine-invariant~\cite{lenglet2006statistics,pennec2006riemannian,moakher2005differential,fletcher2007riemannian}, log-Euclidean~\cite{arsigny2006log,fillard2007clinical}, log-Cholesky~\cite{lin2019riemannian}, Bures-Wasserstein~\cite{bhatia2019bures,takatsu2011wasserstein}, and product metrics with one metric on positive diagonal matrices and one metric on full-rank correlation matrices~\cite{thanwerdas2022theoretically}. Utilizing these metrics, one can analyze full-rank correlation matrices in the space of SPD matrices. However, such approach presents two key limitations: they exclude low-rank PSD matrices and fail to account for constraints on diagonal elements. 

Current literature  explores these two issues separately. One line of research studies the space of PSD matrices of bounded or fixed rank. \cite{bonnabel2010riemannian} viewed the space of PSD matrices of fixed rank as a quotient manifold of the product manifold of  Stiefel manifold and SPD manifold by an orthogonal group. \cite{vandereycken2013riemannian} proposed a homogeneous space geometry for the manifold of PSD matrices of fixed rank, and showed that this geometry led to closed-form geodesics that can be extended indefinitely. \cite{vandereycken2009embedded} treated this space as a submanifold of the real matrices endowed with the usual Euclidean metric, and derived the expressions of the tangent space and geodesics. \cite{journee2010low} developed a low-rank optimization algorithm  using the quotient geometry of the set of PSD matrices of fixed rank. \cite{massart2020quotient} viewed the set of PSD matrices of fixed rank as a quotient manifold of non-compact Stiefel manifold by an orthogonal group. This  manifold is compatible with the Bures-Wasserstein metric. \cite{thanwerdas2023bures} further explored this quotient geometry and revealed that PSD matrices of bounded rank form an orbit space, and the distance is the Bures-Wasserstein distance. Finally, for certain restricted PSD matrices of fixed rank, one can use a  computationally convenient log-Cholesky metric~\cite{neuman2023restricted,chen2023statistical}. 

Another line of research considers the space of full-rank correlation matrices. \cite{archakov2021new} proposed a  parametrization of full-rank correlation matrices using matrix logarithm.  This parametrization has demonstrated applications in econometrics \cite{hafner2023dynamic} and statistics \cite{desai2023connectivity,hu2024applied}, and was  generalized by \cite{zwiernik2023entropic} from the perspective of entropic covariance models. Recently, another Euclidean parametrization for the set of full-rank correlation matrices was proposed by  \cite{thanwerdas2024permutation}. This new parametrization is inverse-consistent. In \cite{lucchetti2023spherical}, the authors analyzed the computational advantages of the spherical parametrization for correlation matrices.  Moreover, \cite{david2019riemannian,thanwerdas2022theoretically} viewed the space of full-rank correlation matrices as a quotient manifold of SPD manifold by the space of positive diagonal matrices, where the group action is given by the matrix congruence. \cite{david2019riemannian} endowed SPD manifold with the affine-invariant metric and obtained the quotient-invariant metric, while \cite{thanwerdas2022theoretically} endowed SPD manifold with permutation-dependent metrics, leading to different quotient metrics.  \cite{thanwerdas2021geodesics} derived the expressions of many fundamental Riemannian operations for the quotient-affine metrics proposed in \cite{david2019riemannian}. Despite the existence of these studies on the geometry of correlation matrices, all the above constructions  require the full-rank assumption, and thus do not generalize to low-rank cases.  

In this paper, we propose a quotient geometry for correlation matrices of bounded or fixed rank, effectively addressing both low-rankness and  unit diagonal constraints. This geometry is rooted in the bijection relationship between bounded-rank correlation matrices and a quotient set of a spherical product manifold by an orthogonal group.  Such bijection was used in \cite{rebonato2000most} for generating random correlation matrices, and was further studied in \cite{rapisarda2007parameterizing,kercheval2008rebonato} from a geometric or topological viewpoint. However, there is no systematic and rigorous study of the quotient geometry of correlation matrices before our work. Now let us delve into the quotient geometry of correlation matrices and describe our contributions. 
Specifically, any $m\times m$ correlation matrix $Z$ of rank at most $k$ can be expressed as $Z=XX^\top$, where $X$ is in  the set $\Pi^m\SSS^{k-1}$ of $m\times k$ matrices with $\ell_2$-normalized rows. Such selection of $X$ is unique up to an orthogonal transformation  \cite[Lemma 2.6]{groetzner2020factorization}, and all possible factors form the following set:
\$
[X]=\{XO\mid O\in\cO(k)\}.
\$
Let 
\$
\Pi^m\SSS^{k-1}/\cO(k)=\{[X]\mid X\in\Pi^m\SSS^{k-1}\}.
\$
Then the map $\Phi_k:[X]\to XX^\top$, is a bijection between the quotient set $\Pi^m\SSS^{k-1}/\cO(k)$ and the set of $m\times m$ correlation matrices of rank at most $k$, $\Corr(m,[k])$. Using $\Phi_k$, we will introduce a quotient geometry on $\Corr(m,[k])$, encoding low-rankness via matrix factorization and unit diagonal entries via spherical product manifolds $\Pi^m\SSS^{k-1}$.

Our first contribution is to construct the quotient structures of bounded or fixed rank correlation matrices. We demonstrate that the following group action 
\$
(X,O)\in\Pi^m\SSS^{k-1}\times\cO(k)\to XO\in\Pi^m\SSS^{k-1}
\$
is smooth, proper, and isometric, where $\Pi^m\SSS^{k-1}$ is endowed with the spherical product metric. As a result, we establish that $\Pi^m\SSS^{k-1}/\cO(k)$ endowed with a quotient distance is a Riemannian orbit space~\cite{michor2003riemannian}, and hence a geodesic metric space. Using $\Phi_k$, we can endow $\Corr(m,[k])$ with a distance such that $\Phi_k$ is distance-preserving. Moreover, we demonstrate that this metric structure induces the Euclidean subspace topology. 

Then, we examine the orbit stratification of the orbit space $\Pi^m\SSS^{k-1}/\cO(k)$. Specifically, we demonstrate that the stratification is determined by the rank of the matrices. Each stratum is given by $\Pi^m_\ell\SSS^{k-1}/\cO(k)$, where $\Pi^m_\ell\SSS^{k-1}$ is a submanifold of $\Pi^m\SSS^{k-1}$ with rank $\ell$. This is in bijection with the set $\Corr(m,\ell)$ of  $m\times m$ correlation matrices with rank $\ell$. In addition, we show that both $\Pi^m_\ell\SSS^{k-1}/\cO(k)$ and $\Corr(m,\ell)$ are smooth manifolds, and the map
\$
\Phi_{k,\ell}:\Pi^m_\ell\SSS^{k-1}/\cO(k)\to\Corr(m,\ell),\quad [X]\to XX^\top
\$
is a diffeomorphism. This is called the compatibility of the smooth structure. 

Moreover, we show that the principal stratum, $\Pi^m_k\SSS^{k-1}/\cO(k)$, of this orbit space has a compatible Riemannian structure. The key is to show that the group action 
\$
(X,O)\in\Pi^m_k\SSS^{k-1}\times\cO(k)\to XO\in\Pi^m_k\SSS^{k-1}
\$
is smooth, proper, free, and isometric. These findings bear resemblance to the Bures-Wasserstein geometry on bounded-rank covariance matrices~\cite{thanwerdas2023bures}. However, we remark a significant difference: the quotient distance on $\Corr(m,[k])$ depends on $k$, while the quotient distance on the set $\Cov(m,[k])$ of $m\times m$ covariance matrices of rank at most $k$, does not depend on $k$. This is due to the non-linearity of the spherical distance in contrast to the Euclidean distance, and it emphasizes the importance of choosing an appropriate $k$ in practice. By default, we recommend $k=m$. 

Our second contribution is to investigate geometric properties and computations in the orbit space $\Pi^m\SSS^{k-1}/\cO(k)$. The distance function does not admit a closed form solution, so we propose a Riemannian gradient descent algorithm for computing the distance. Also, it is interesting to study the weighted \Frechet mean of samples in a metric space. We develop an alternating minimization algorithm for the computation. Furthermore, we demonstrate that any minimizing geodesic has constant rank on the interior of the segment. 

Our third contribution is to examine the geometric quantities on the Riemannian quotient manifold $\Pi^m_k\SSS^{k-1}/\cO(k)$. Specifically, we describe the horizontal and vertical spaces and projections. The Riemannian quotient metric is introduced and we prove that the induced Riemannian manifold has an injectivity radius of zero. We show that the sectional curvatures of this quotient manifold are non-negative. We also provide the expressions of exponential map, gradient, connection, and Hessian.  Moreover, we present an approach to compute a Riemannian logarithm in the manifold. The idea is to address the alignment problem encountered during the computation of distances. 

Our last contribution is to provide an application of our geometric approach to a brain functional connectivity study using the Autism Brain Imaging Data Exchange. The code is available at \href{https://github.com/HengchaoChen/geomcorr}{https://github.com/HengchaoChen/geomcorr}. In our study, we find that the correlation matrices of the time series data from one single site tend to be of fixed rank. This justifies the usage of our proposed geometry in studying the difference between connectivity matrices of the patients and normal controls. 

The rest of the paper proceeds as follows. We  introduce notations used throughout the paper in Section \ref{sec:notation}. In Section \ref{sec:2}, we review basic concepts in geometry. In Section \ref{sec:3}, we introduce a series of quotient structures for the sets of correlation matrices. We also examine the compatibility of the topological, smooth, and Riemannian structures. In Section \ref{sec:4}, we delve into the orbit space structure of correlation matrices of bounded rank. Distances, geodesics, and the \Frechet mean are explored. In Section \ref{sec:5}, we study the Riemannian quotient manifold structure of correlation matrices of fixed rank. We discuss various geometric quantities.  Section \ref{sec:app} presents our application in the autism study, and Section \ref{sec:conclusion} gives concluding discussions. More geometric tools and omitted proofs are provided in the Appendix.

\subsection{Notation}\label{sec:notation} In this paper, we use the following notations. We use $\zero$ to denote the null matrix whose size depends on the context. $I_k$ denotes the identity matrix of size $k$. For two symmetric matrices $\Sigma,\Lambda$, we denote $\Sigma\succeq\Lambda$ when $\Sigma-\Lambda$ is a positive semi-definite matrix. For vectors $v,w\in\RR^m$, $\norm{v}$ denotes the $\ell_2$ norm of $v$ and $\inner{v}{w}$ denotes the inner product of $v$ and $w$. In addition, we use $\diag(v)\in\RR^{m\times m}$ to denote the diagonal matrix with diagonal elements $\{v_i\}$. For a matrix $X\in\RR^{m\times k}$, we use $X_{i\cdot}$ to denote the $i$-th row of $X$. In this paper, we will treat $X_{i\cdot}$ as a column vector. For a symmetric matrix $Z\in\RR^{m\times m}$, denote by $\Diag(Z)$ the diagonal matrix in $\RR^{m\times m}$ with diagonal elements equal to those of $Z$. With some abuse of notation, we use  $\diag(Z)\in\RR^m$ to denote an $m$-dimensional vector whose elements equal to the diagonal elements of $Z$. The meaning of $\diag(\cdot)$ thus depends on whether its input is a vector or a symmetric matrix. 
For a metric space $(\cX,d)$ and $\{x_n\},x\in\cM$, we use 
$x_n\to x\textnormal{ in }d,$ to indicate that $x_n$ converges to $x$ in terms of the distance $d$. Finally, we use the following notations for matrix spaces. Let $m,k,\ell\in\NN$ and $\ell\leq k\leq m$. 
\vspace{3pt}
{
\renewcommand{\labelitemi}{\scalebox{0.4}{$\bullet$}}
\begin{itemize}
    
    \item the unit $k$-sphere $\SSS^{k}=\{x\in \RR^{k+1}\mid \norm{x}=1\}$, where $\norm{\cdot}$ is the $\ell_2$ norm;

    \item the space of $m\times k$ matrices $\RR^{m\times k}$;
    
    \item the space of $m\times k$ matrices of rank $\ell$, $\RR^{m\times k}_\ell$;

    \item the space of $m\times k$ matrices with $\ell_2$-normalized rows:
    \$
    \Pi^m\SSS^{k-1}=\{X\in\RR^{m\times k}\mid \norm{X_{i\cdot}}=1,\forall i\leq m\};
    \$
    
    \item the space of  $m\times k$ matrices with $\ell_2$-normalized rows and rank $\ell$:
    \$
    \Pi_\ell^m\SSS^{k-1}=\{X\in \Pi^m\SSS^{k-1}\mid \rank(X)=\ell\};
    \$
    
    \item the space of $m\times m$ covariance matrices:
    \$
    \Cov(m)=\{Z\in\RR^{m\times m}\mid Z\succeq \zero\};
    \$

    \item the space of $m\times m$ covariance matrices of rank at most $k$:
    \$
    \Cov(m,[k])=\{Z\in\Cov(m)\mid \rank(Z)\leq k\};
    \$
    In particular, $\Cov(m)=\Cov(m,[m])$.

    \item the space of $m\times m$ covariance matrices of fixed rank $k$:
    \$
    \Cov(m,k)=\{Z\in\Cov(m)\mid \rank(Z)=k\};
    \$
    
    \item the space of $m\times m$ correlation matrices:
    \$
    \Corr(m)=\{Z\in\Cov(m) \mid  Z_{ii}=1,\forall i\leq m\};
    \$

    \item the space of $m\times m$ correlation matrices of rank at most $k$:
    \$
    \Corr(m,[k])=\{Z\in\Cov(m,[k])\mid  Z_{ii}=1,\forall i\leq m\};
    \$
    In particular, $\Corr(m)=\Corr(m,[m])$.
    
    \item the space of $m\times m$ correlation matrices of fixed rank $k$:
    \$
    \Corr(m,k)=\{Z\in\Cov(m,k)\mid  Z_{ii}=1,\forall i\leq m\}.
    \$
    
    \item the set of $k\times k$ skew-symmetric matrices, $\cS_{\skew}(k)$;

    \item the orthogonal group $\cO(m)$;

    \item the Stiefel manifold $\St(m,k)=\cO(m)/\cO(m-k)$;
\end{itemize}
}

\section{Preliminary}\label{sec:2}

In this section, we review basic geometric concepts. We review basic concepts in metric geometry, including length and geodesic, in Section \ref{sec:2.1}. More details in metric geometry can be found in the textbook \cite{burago2022course}. Next, in Section \ref{sec:2.2}, we review basic concepts in Riemannian geometry,  such as Riemannian metric, distance, geodesic, exponential and logarithmic map. For more details in Riemannian geometry, one can refer to \cite{do1992riemannian,petersen2006riemannian}. In our review, familiarity with smooth manifolds is presumed. One can refer to the textbook \cite{lee2012introduction} for an introduction to smooth manifolds.  Finally, in Section \ref{sec:2.3}, we review quotient metric spaces and two examples, quotient manifolds and orbit spaces. We will use them to examine the quotient geometry of correlation matrices.  We give their definitions and basic properties, such as the orbit stratification and the convexity properties in orbit spaces. Our introduction to quotient manifolds is based on \cite{absil2008optimization,lee2012introduction,thanwerdas2023bures} and our introduction to orbit spaces is based on \cite{thanwerdas2023bures,michor2003riemannian,michor1997isometric}. It is  worth noting that our presentation closely follows Section 2 of \cite{thanwerdas2023bures}.  One can refer to these literature for more details. 

\subsection{Metric spaces}\label{sec:2.1}

In this section, we give a brief introduction to metric space, following  \cite[Section 2.1]{thanwerdas2023bures}. Let $(\cM,d)$ be a metric space. A curve in $\cM$ is a continuous map $c:I\to\cM$, where $I$ is an interval of $\RR$. For any $[a,b]\subseteq I$, we define $c|_{[a,b]}$ as the curve $c$ constrained to $[a,b]$. The length of $c|_{[a,b]}$ is defined as 
\$
L(c|_{[a,b]})\coloneqq \sup\sum_{k=0}^sd(c(t_k),c(t_{k+1})),
\$
where the supremum is taken over all subdivisions $a=t_0\leq t_1\leq \ldots\leq t_{s+1}=b$. The curve $c$ is rectifiable when $L(c|_{[a,b]})$ is finite for all $[a,b]\subseteq I$. For any $x,y\in\cM$, the length distance between $x$ and $y$ is defined by $d_L(x,y)=\inf_c L(c)$ over all rectifiable curves $c$ connecting $x$ and $y$. The length distance $d_L$ is a distance when $\cM$ is connected by rectifiable curves. When $d_L=d$, $(\cM,d)$ is called a length space. 

Let $c:I\to\cM$ be a curve in $\cM$. $c$ is said to be parameterized at constant speed when there exists $v\geq 0$ such that $L(c|_{[a,b]})=v(b-a)$ for all $a\leq b$. When $v=1$, we say $c$ is parameterized at unit speed. The curve $c$ is said to be (globally) minimizing if $L(c|_{[a,b]})=d(c(a),c(b))$ for all $a\leq b$ in $I$. $c$ is locally minimizing if for every $t\in I$, there is a neighborhood $I_0\subseteq I$ of $t$ such that $c|_{I_0}$ is minimizing. A locally minimizing curve of constant speed is called a geodesic. A globally minimizing curve of constant speed is called a minimizing geodesic. A geodesic metric space is a length space such that any two points in the space are connected by a minimizing geodesic. 

\subsection{Riemannian manifolds}\label{sec:2.2}

In this section, we briefly introduce Riemannian manifolds, following \cite{do1992riemannian,petersen2006riemannian,thanwerdas2023bures}.
A Riemannian manifold $(\cM,g)$ is a smooth manifold $\cM$ endowed with a smoothly varying family of inner products $g_x:T_x\cM\times T_x\cM\to\RR$, where $T_x\cM$ is the tangent space of $\cM$ at $x\in\cM$. $g$ is called a Riemannian metric and it allows us to measure geometric quantities on $\cM$ such as distances and curvatures. For a tangent vector $v\in T_{x}\cM$, its norm is given by $\norm{v}=g^{1/2}(v,v)$. Given a piecewise smooth curve $c$ in $\cM$, its length is given by integrating the norm of tangent vectors along the curve, i.e., $L(c)=\int \norm{\dot{c}(t)}dt$. Given any two points $x,y\in\cM$, the distance  between them is $d(x,y)=\inf_c L(c)$ over all piecewise smooth curves $c$ that connect $x$ and $y$. When $\cM$ is a connected manifold, the function $d$ is a distance, and the metric space $(\cM,d)$ is a length space. We say a connected manifold $(\cM,g)$ is a complete manifold if and only if it is complete as a metric space $(\cM,d)$. 

Let $(\cM,g)$ be a connected Riemannian manifold and $d$ the distance. A curve in $\cM$ is a geodesic (or minimizing geodesic) if it is a geodesic (or minimizing geodesic) in the metric space $(\cM,d)$. By the Hopf-Rinow theorem, $\cM$ is complete if and only if all geodesics extend indefinitely. Moreover, when $\cM$ is complete, any two points in $\cM$ can be connected by a minimizing geodesic. It implies that a connected complete manifold is a geodesic metric space. 

Let $(\cM,g)$ be a connected Riemannian manifold and $d$ be the distance. For any tangent vector $v\in T_{x}\cM$, there is a unique geodesic $\gamma_v(t)$ with $x$ as its initial location and $v$ as its initial velocity. The definition interval $I_{x,v}\subseteq\RR$ of $\gamma_v(t)$ is the maximal interval of $\RR$ on which the geodesic $\gamma_v(t)$ is defined. $I_{x,v}$ is $\RR$ when $\cM$ is complete.  The exponential map $\Exp_x:T_x\cM\to\cM$ is defined by $\Exp_x(v)=\gamma_v(1)$ when $\gamma_v(1)$ exists. The definition domain of $\Exp_x$ is 
\$
\cD_x=\{v\in T_x\cM\mid 1\in I_{x,v}\}.
\$
The cut time at $x\in\cM$ in the direction $v\in T_x\cM$, $\norm{v}=1$, is defined by 
\$
t_{\rm cut}(x,v)=\sup\{t\in I_{x,v}\mid d(x,\Exp_x(tv))=t \}\in (0,+\infty].
\$
The injection domain of $\Exp_x$ is the set 
\$
\Inj(x)=\{tv\mid t\in[0,t_{\rm cut}(x,v),v\in T_x\cM, \norm{v}=1\}\subseteq \cD_x.
\$
The injectivity radius at $x\in\cM$ is defined by 
\$
\inj(x)=\sup\{\epsilon\mid \cB_x(\epsilon)\subseteq \Inj(x)\},
\$
where $\cB_x(\epsilon)=\{v\in T_{x}\cM\mid \norm{x}<\epsilon\}$ is the centered open ball in $T_x\cM$ of radius $\epsilon$. The set of preimages of $y\in\cM$ from $x$ is 
\$
\Pre_x(y)=\Exp_x^{-1}(y)=\{v\in T_x\cM\mid 1\in I_{x,v},\Exp_x(v)=y\}.
\$
A logarithm of $y$ from $x$ is a preimage $v\in\Pre_x(y)$ of $y$ from $x$ such that $\norm{v}=d(x,y)$. We use $\Log_x(y)$ to denote the set of all logarithms of $y$ from $x$. Let $\cU_x\subseteq\cM$ the set of $y$ such that $\Log_x(y)$ consists of a single point. This defines a map $\Log_x:\cU_x\to T_x\cM$, which we call the logarithmic map. Remark that $\cB_x(\inj(x))\subseteq\cU_x$ and for all $\epsilon<\inj(x)$, $\Log_x:\Exp_x(\cB_x(\epsilon))\to\cB_x(\epsilon)$ is a diffeomorphism.

\subsection{Quotient spaces}\label{sec:2.3}

In this section, we  introduce quotient spaces and our contents follow the structure of \cite[Section 2.2]{thanwerdas2023bures}.
Let $(\cM,d)$ be a metric space and $G$ be a group acting isometrically on $(\cM,d)$. Define the quotient space $\cM^0=\cM/G$ and $d^0:\cM^0\times \cM^0\to\RR$ as 
\$
d^0([x],[y])=\inf_{a\in G}d(ax,y)\in[0,+\infty).
\$
Here $[x]=\{ax\mid a\in G\}$ is an element in the quotient space $\cM^0$ and $ax$ is given by the action of $G$ on $\cM$.  
If the orbits $\{[x]\}_{x\in\cM}$ are closed, then $d^0$ defines a distance on $\cM^0$ called the quotient distance, and $(\cM^0,d^0)$ is called a quotient metric space. Proposition \ref{prop:1} relates the curve lengths  in the quotient metric space $(\cM^0,d^0)$ to the curve lengths in $(\cM,d)$. Note that two points $x,y\in\cM$ are said to be registered when $d(x,y)=d^0(Gx,Gy)$. 

\begin{proposition}[\cite{thanwerdas2023bures}, Lemma 2.13]\label{prop:1}
    Let $L$ be the length on both $\cM$ and $\cM^0=\cM/G$, and $\pi:\cM\to\cM^0$ be the canonical projection. For all curves $c:[0,1]\to\cM$, we have $L(\pi\circ c)\leq L(c)$. Furthermore, if $x=c(0)$ and $y=c(1)$ are registered and if $c$ is minimizing, then $\pi\circ c$ is minimizing and $L(\pi\circ c)=L(c)=d(x,y)=d^0(\pi(x),\pi(y))$.
\end{proposition} 

In the sequel, we introduce two common examples of quotient metric spaces. The first example is called (Riemannian) quotient manifolds and the second is called (Riemannian) orbit spaces. These examples are both quotient spaces of some Riemannian manifold $(\cM,g)$ by a certain Lie group $G$. The difference is that in the case of quotient manifolds, the Lie group must act freely on the Riemannian manifolds, while in the case of orbit spaces, we do not need this condition. 

\subsubsection{Quotient manifolds}\label{sec:2.3.1}

Let $(\cM,g)$ be a Riemannian manifold and $G$ a Lie group acting smoothly, properly, freely, and isometrically on $(\cM,g)$. Then there exists a unique smooth structure on $\cM^0=\cM/G$ such that the quotient map $\pi:\cM\to\cM^0$ is a smooth submersion. The dimension of $\cM^0$ is $\dim(\cM^0)=\dim(\cM)-\dim(G)$. For all $x\in\cM$, we define the vertical space at $x$ as 
\$
\cV_x=T_x\cM_x=\ker d_x\pi,
\$
where $\cM_x=\pi^{-1}(x)$ is a submanifold of $\cM$. The horizontal space at $x$ is $\cH_x=\cV_x^{\perp}$ so that $T_x\cM=\cV_x\oplus\cH_x$. For any $v\in T_{\pi(x)}\cM^0$, there exists a unique vector $v^{\sharp}_x\in \cH_x$ such that $d_x\pi(v^{\sharp}_x)=v$. We call $v^{\sharp}_x$ as a horizontal lift of $v$ at $x$. The quotient metric on $\cM^0$ is defined as follows:
\$
g^0_{\pi(x)}(v,w)=g_x(v^{\sharp}_x,w^{\sharp}_x), \quad \forall v,w\in T_{\pi(x)}\cM^0.
\$
Equipped with this metric, $(\cM^0,g^0)$ becomes a Riemannian manifold, named quotient manifold. The quotient map $\pi$ is a Riemannian submersion from $(\cM,g)$ to $(\cM^0,g^0)$. The  distance $d^0$ induced by the quotient metric $g^0$ is the quotient distance of $d$ induced by $g$, which is 
\$
d^0(\pi(x),\pi(y))=\inf_{a\in G}d(ax,y).
\$
The induced space $(\cM^0,d^0)$ is a special case of quotient metric spaces. 
Fundamental properties of geodesics of a quotient metric are presented in Proposition~\ref{prop:quotient-geodesic}. 

\begin{proposition}[\cite{massart2020quotient}, Theorem 2.15; \cite{o1966fundamental}]\label{prop:quotient-geodesic}
     Let $x\in\cM$ and $v\in T_{\pi(x)}\cM^0$. We have $I_{x,v_x^\sharp}\subseteq I_{\pi(x),v}$, and for all $t\in I_{x,v_x^\sharp}$, we have $\Exp_{\pi(x)}(tv)=\pi(\Exp_x(tv_x^\sharp))$. 
\end{proposition}

\subsubsection{Orbit spaces}\label{sec:2.3.2}

Let $(\cM,g)$ be a connected complete Riemannian manifold and $G$ be a Lie group acting smoothly, properly, and isometrically on $(\cM,g)$. Then $(\cM/G,d^0)$, called the orbit space of $\cM$ with respect to $G$, is a complete metric space and a length space, where $d^0$ is the quotient distance. An important property of the orbit space $(\cM/G,d^0)$ is the orbit stratification, which is summarized below. 

\begin{proposition}[The properties of orbit spaces \cite{michor2003riemannian}]\label{prop:2}
    Let $(\cM,g)$ be a connected complete Riemannian manifold, and $G$ be a Lie group acting smoothly, properly, and isometrically on $(\cM,g)$. Let $\pi:\cM\to\cM/G$ be the canonical projection. For a Lie subgroup $H$ of $G$, define $(H)=\{aHa^{-1}\mid a\in G\}$ the conjugacy class of $H$. Let $\cM_{(H)}$ be the set of points $x\in\cM$ such that the stabilizer of $x$, ${\rm Stab}(x)=\{g\in G\mid gx=x\}$, belongs to $(H)$. The following properties hold. 
    \begin{enumerate}[leftmargin = 32pt]
        
        \item $\cM_{(H)}$ is a smooth embedded submanifold of $\cM$. 
        \item Denote $(\cM/G)_{(H)}=\pi(\cM_{(H)})=\cM_{(H)}/G$ the isotropy stratum of type $(H)$. Then $\pi_{(H)}\coloneqq \pi|_{\cM(H)}:\cM_{(H)}\to (\cM/G)_{(H)}$ is a smooth fiber bundle with fiber type $G/H$. In particular, the projection $\pi_{(H)}$ is a smooth submersion.
        \item The isotropy strata form a partition of $\cM/G$.
        \item There exists a unique minimal isotropy group $(H_{\rm reg})$ such that
        {
        \renewcommand{\labelitemi}{\scalebox{0.4}{$\bullet$}}
        \begin{itemize}
            \item $(\cM/G)_{(H_{\rm reg})}$ is connected, locally connected, open, and dense in $\cM/G$;
            \item The slice representations at all points of $\cM_{(H_{\rm reg})}$ are trivial;
            \item $\dim(\cM/G)_{(H_{\rm reg})}=\dim \cM-\dim G+\dim H_{\rm reg}$.
        \end{itemize}
        }
        The stratum $(\cM/G)_{(H_{\rm reg})}$ is called the principal isotropy stratum. 
    \end{enumerate}
\end{proposition}

Another property in an orbit space is that minimizing geodesics in $(\cM/G,d^0)$ can hit strata which are more singular only at the end points. This is formally presented in Proposition \ref{prop:3} and is phrased as the convexity property in \cite{michor2003riemannian}. 

\begin{proposition}[Convexity property, \cite{michor2003riemannian}]\label{prop:3}
    Suppose the assumptions in Proposition \ref{prop:2} hold. Given Lie subgroups $H_1,H_2$ of $G$, denote $(H_1)\leq (H_2)$ if $H_1$ is conjugated to a subgroup of $H_2$, i.e., there exists a $a\in G$ such that $aH_1a^{-1}$ is a subgroup of $H_2$. Let $\gamma:[0,1]\to\cM/G$ be a minimizing curve. For $t\in[0,1]$, denote by $(\cM/G)_{(H_t)}$ the stratum of $\gamma(t)$. Then, for all $t\in(0,1)$, $(H_t)\leq (H_0)$ and $(H_t)\leq (H_1)$. 
\end{proposition}

\section{Quotient geometry of correlation matrices}\label{sec:3}

In this section, we introduce  quotient structures for different sets of correlation matrices. Specifically, we show that the set $\Corr(m,[k])$ of correlation matrices of rank at most $k$ is endowed with an orbit space structure. This is rooted in the bijection between $\Corr(m,[k])$ and a quotient set $\Pi^m\SSS^{k-1}/\cO(k)$, as discussed in the introduction. The bijective map is
\#\label{equ:Phi_k}
\Phi_k:\Pi^m\SSS^{k-1}/\cO(k)\to\Corr(m,[k]),\quad [X]\to XX^\top.
\#
We endow $\Pi^m\SSS^{k-1}$ with the spherical product metric and show that $\Pi^m\SSS^{k-1}/\cO(k)$ has an orbit space structure, with its distance denoted by $d^{\orb,k}$. Then we define the metric structure on $\Corr(m,[k])$ via the map $\Phi_k$ such that $\Phi_k$ is distance-preserving. With some abuse of notation, we still denote this distance as $d^{\orb,k}$. As we will show, the metric space $(\Corr(m,[k]),d^{\orb,k})$ shares the same topology as $(\Corr(m,[k]),d^{\rm E})$, where $d^{\rm E}$ is the Euclidean distance. 

Next, we study the orbit stratification of $\Pi^{m}\SSS^{k-1}/\cO(k)$. We show that the strata are determined by the rank of matrices. In particular, any stratum of the orbit space $\Pi^{m}\SSS^{k-1}/\cO(k)$ is given by $\Pi^m_\ell\SSS^{k-1}/\cO(k)$ for some rank $\ell\leq k$. The map
\#\label{equ:Phi_kl}
\Phi_{k,\ell}:\Pi^m_\ell\SSS^{k-1}/\cO(k)\to\Corr(m,\ell),\quad [X]\to XX^\top,
\#
is a bijection between this stratum and $\Corr(m,\ell)$. Also, we show that $\Pi^m_\ell\SSS^{k-1}/\cO(k)$ is a smooth manifold, $\Corr(m,\ell)$ is a smooth embedded submanifold of $\RR^{m\times m}$, and the map $\Phi_{k,\ell}$ is indeed a diffeomorphism between these two smooth manifolds. 

Then, we study the principal stratum $\Pi^m_k\SSS^{k-1}/\cO(k)$ of $\Pi^m\SSS^{k-1}/\cO(k)$. It is dense in the orbit space. It is diffeomorphic to a smooth submanifold $\Corr(m,k)$ of $\RR^{m\times m}$, where $\Phi_{k,k}$ gives the diffeomorphism. More importantly, we show that $\Pi^m_k\SSS^{k-1}/\cO(k)$ admits a quotient manifold structure with the Riemannian metric denoted by $g^{\qt,k}$. Since $\Phi_{k,k}$ is a diffeomorphism, we can define a Riemannian metric $g^{k}$ on $\Corr(m,k)$ such that $\Phi_{k,k}$ is a Riemannian isomorphism. We show that the distance induced by such Riemannian metric $g^{\qt,k}$ or $g^{k}$ is equal to the orbit space distance $d^{\orb,k}$ in the corresponding spaces. 

Figure \ref{fig:1} summarizes the connections of the above structures.

\begin{figure}[t]
    \centering
    \includegraphics[width=\textwidth]{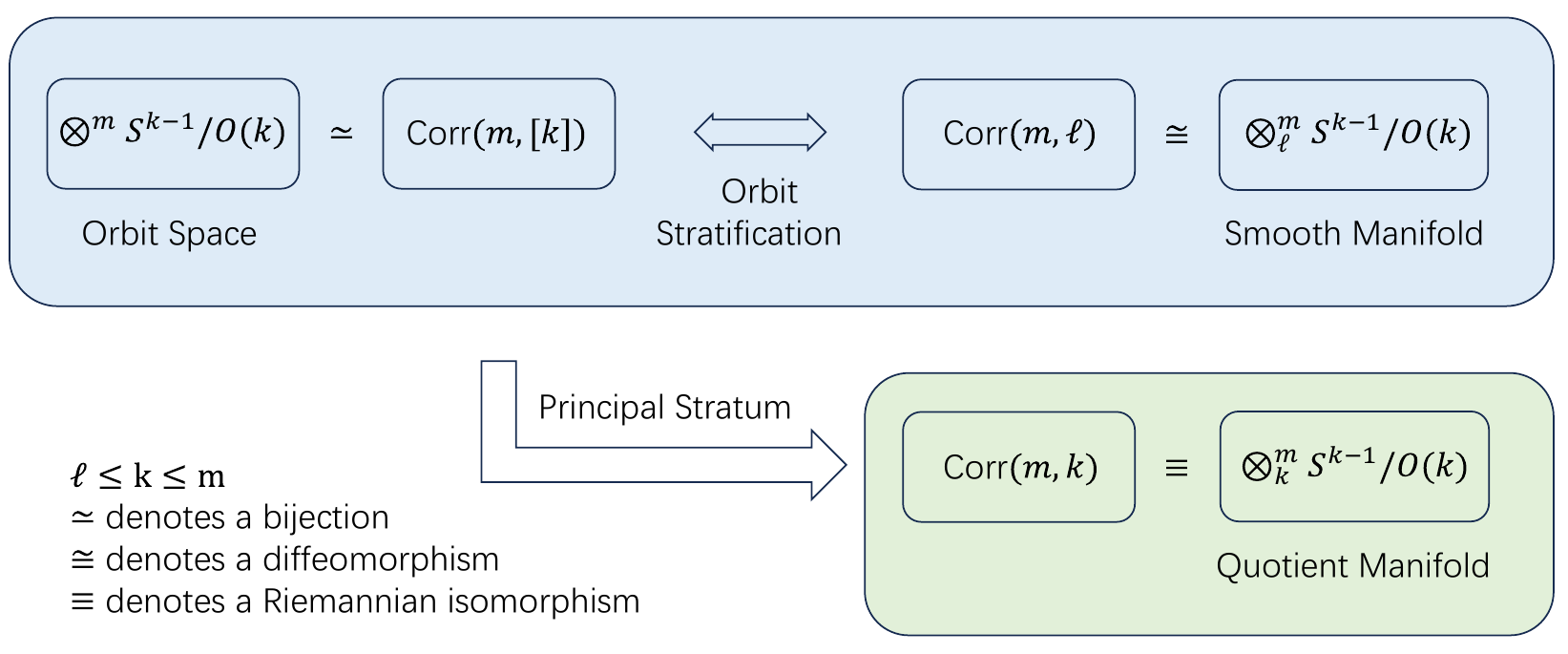}
    \caption{Quotient structures of correlation matrices of bounded rank. It shows that $\Corr(m,[k])$ is in bijection with the orbit space $\Pi^m\SSS^{k-1}/\cO(k)$. The strata of the orbit space are given by $\Corr(m,\ell)$ for some $\ell\leq m$. The principal stratum $\Corr(m,k)$ has a compatible Riemannian quotient structure.}
    \label{fig:1}
\end{figure}

Finally, we emphasize a subtle point. By the above discussion, we can construct a series of orbit spaces $(\Corr(m,[k]),d^{\orb,k})$ for different $k\leq m$. As we will show, the topologies of these metric spaces coincide with the subspace topologies induced from $\RR^{n\times n}$. The strata of these orbit spaces are  $\Corr(m,\ell)$ for some $\ell$, and can be smoothly embedded into $\RR^{m\times m}$. However, despite these compatible topological and smooth structures, the distances $d^{\orb,k}$  vary for different $k$, even when the comparison is taken in the same domain. For this reason, we need to select $k$ at the outset for applications involving the metric structures. We suggest $k=m$ by default.

In the remainder of this section, we elaborate how to construct the above quotient structures. In Section \ref{sec:3.1}, we construct the orbit space structure of $\Corr(m,[k])$. The orbit stratification is discussed in Section \ref{sec:3.2}. We introduce the Riemannian structure of $\Corr(m,k)$ and discuss its compatibility with the orbit space $\Corr(m,[k])$ in Section \ref{sec:3.3}. In Section \ref{sec:3.4}, we show that $d^{\orb,k}$ vary for different $k$, even when restricted to the same domain. In Section \ref{sec:4} and Section \ref{sec:5}, we will provide more detailed geometric properties and computations regarding these structures.

\subsection{An orbit space structure of $\Corr(m,[k])\simeq\Pi^m\SSS^{k-1}/\cO(k)$}\label{sec:3.1}

In the introduction, we have shown that $\Corr(m,[k])$ is in bijection with $\Pi^m\SSS^{k-1}/\cO(k)$ with $\Phi_k$ in \eqref{equ:Phi_k} being the bijective map. Now we endow $\Pi^m\SSS^{k-1}$ with the natural spherical product metric $g^{\Pi\SSS,k}$.  The spherical product manifold $(\Pi^m\SSS^{k-1},g^{\Pi\SSS,k})$ is a connected complete Riemannian manifold when\footnote{Here we consider $k\geq 2$. When $k=1$, the space consists trivially of several discrete points.} $k\geq 2$. The induced geodesic distance is
\#\label{equ:spherical-dist}
d^{\Pi\SSS,k}(X,Y)=\sqrt{\sum_{i=1}^m\arccos^2(\inner{X_{i\cdot}}{Y_{i\cdot}})},\quad\forall X,Y\in\Pi^m\SSS^{k-1}.
\#
Let the group action of $\cO(k)$ on the manifold $(\Pi^m\SSS^{k-1},g^{\Pi\SSS,k})$ be as follows:
\#\label{equ:action}
(X,O)\in\Pi^m\SSS^{k-1}\times \cO(k)\to XO\in\Pi^m\SSS^{k-1}.
\#
Proposition \ref{prop:4} shows that this group action is smooth, proper, and isometric.

\begin{proposition}\label{prop:4}
    The group action \eqref{equ:action} is smooth, proper, and isometric. 
\end{proposition}

\begin{proof}
    It is direct to see that the group action \eqref{equ:action} is smooth. Proposition \ref{prop:a2} states that every continuous action by a compact Lie group on a manifold is proper. Since $\cO(k)$ is a compact Lie group, the action \eqref{equ:action} is proper. Since the action of multiplying an orthogonal matrix preserves the spherical product metric, the action \eqref{equ:action} is isometric. This proves the proposition. 
\end{proof}

Thus, \eqref{equ:action} is a smooth, proper, and isometric Lie group action on the connected Riemannian manifold $(\Pi^m\SSS^{k-1},g^{\Pi\SSS,k})$ when $k\geq 2$.  As discussed in Section \ref{sec:2.3.2}, the quotient space $\Pi^{m}\SSS^{k-1}/\cO(k)$, when endowed with the quotient distance $d^{\orb,k}$, forms an orbit space. The quotient distance $d^{\orb,k}$ is 
\#\label{equ:dist}
d^{\orb,k}([X],[Y])=\inf_{O\in\cO(k)}d^{\Pi\SSS,k}(X,YO),\quad\forall [X],[Y]\in\Pi^m\SSS^{k-1}/\cO(k).
\#
In particular, $(\Pi^{m}\SSS^{k-1}/\cO(k),d^{\orb,k})$ is a geodesic metric space. Since $\Corr(m,[k])$ is in bijection with this metric space, we can define a distance on $\Corr(m,[k])$ such that the bijection $\Phi_k$ is distance-preserving. With some abuse of notation\footnote{One can distinguish the meaning of $d^{\orb,k}$ according to the context.}, we still denote this distance as $d^{\orb,k}$, which is
\#\label{equ:dist-corr}
d^{\orb,k}(XX^\top,YY^\top)=d^{\orb,k}([X],[Y]),\quad\forall X,Y\in\Pi^{m}\SSS^{k-1}.
\#
The metric space $(\Corr(m,[k]),d^{\orb,k})$ is an orbit space.
Proposition \ref{prop:topology} demonstrates that this metric space shares the same topology as the metric space $(\Corr(m,[k]),d^{\rm E})$, where $d^{\rm E}$ is the Euclidean distance.

\begin{proposition}\label{prop:topology}
    The metric spaces $(\Corr(m,[k]),d^{\orb,k})$ and  $(\Corr(m,[k]),d^{\rm E})$ share the same topology. 
\end{proposition}

\begin{proof}
    The proof is provided in the Appendix \ref{sec:topology}.
\end{proof}

\subsection{Stratification}\label{sec:3.2}

Orbit stratification is a significant feature of an orbit space. This section reveals that the strata of the orbit space $\Pi^m\SSS^{k-1}/\cO(k)$ are classified by the rank of matrices. To proceed, we study the stabilizer (or the isotropy group) of a matrix $X\in\Pi^m\SSS^{k-1}$, defined by $\Stab(X)=\{O\in\cO(k)\mid XO=X\}$. By the singular value decomposition, we can write 
\$
X=PJ_{k,\ell}Q, \quad J_{k,\ell}=\begin{pmatrix}
    \diag(\sigma_1,\ldots,\sigma_\ell) & \zero \\
    \zero & \zero
\end{pmatrix}\in\RR^{m\times k},
\$
where $\ell=\rank(X)$ is the rank of $X$, $\sigma_1\geq\ldots\geq\sigma_\ell>0$, $P\in\cO(m)$, and $Q\in\cO(k)$. The stabilizer $H_{k,\ell}$ of the matrix $J_{k,\ell}$ is given in  \cite[Section 3.1]{thanwerdas2023bures}.
The stabilizer of $X$ is $\Stab(X)=Q^{-1}H_{k,\ell}Q$. In particular, two matrices have conjugate stabilizers if and only if they have the same rank. Moreover, $(H_{k,\ell_1})\leq (H_{k,\ell_2})$ if and only if $\ell_1\geq \ell_2$. 

Since the orbit strata of an orbit space are the sets of points that have conjugate stabilizers~\cite{michor2003riemannian}, the orbit strata of $\Pi^m\SSS^{k-1}/\cO(k)$ are given by the sets of points that have the same rank. In particular, each stratum is represented by $\Pi_\ell^m\SSS^{k-1}/\cO(k)$ for some $\ell\leq k$. It is in bijection with the set $\Corr(m,\ell)$ through the map $\Phi_{k,l}$ in \eqref{equ:Phi_kl}. The principal (regular) stratum is given by $\Pi^m_k\SSS^{k-1}/\cO(k)$ or equivalently $\Corr(m,k)$. It is an open and dense subset in the orbit space. 

Let us examine the smooth structures of the orbit strata. Specifically, $\Pi_\ell^m\SSS^{k-1}$ is a smooth submanifold of $\Pi^m\SSS^{k-1}$. The stratum $\Pi_\ell^m\SSS^{k-1}/\cO(k)$ has a smooth manifold structure. In addition, the quotient map
\#\label{equ:pi_kl}
\pi_{k,\ell}\coloneqq\pi_{(H_{k,\ell})}:\Pi^m_\ell\SSS^{k-1}\to \Pi^m_\ell\SSS^{k-1}/\cO(k)
\#
is a smooth fiber bundle with fiber type $\St(k,\ell)$. In particular, it is a smooth submersion. Proposition \ref{prop:smooth} demonstrates that $\Corr(m,\ell)$ is a smooth embedded submanifold of $\RR^{m\times m}$, and the bijection $\Phi_{k,\ell}$ in \eqref{equ:Phi_kl} is indeed a diffeomorphism between two smooth manifolds.  

\begin{proposition}\label{prop:smooth}
    $\Corr(m,\ell)$ is a smooth embedded submanifold of $\RR^{m\times m}$. The bijection $\Phi_{k,\ell}$ in \eqref{equ:Phi_kl} is a diffeomorphism between two smooth manifolds. 
\end{proposition}

\begin{proof}
    The proof is provided in the Appendix \ref{sec:b1}.
\end{proof}

\subsection{Riemannian structure of $\Corr(m,k)$ and compatibility}\label{sec:3.3}

In this section, we examine the Riemannian structure of the principal stratum $\Pi_k^m\SSS^{k-1}/\cO(k)$ of the orbit space $\Pi^m\SSS^{k-1}/\cO(k)$. In the previous section, we showed that it is a dense set in the orbit space, is a smooth manifold, and is diffeomorphic to a smooth submanifold $\Corr(m,k)$ of $\RR^{m\times m}$. In this section, we show that this principal stratum has a natural Riemannian quotient manifold structure, and the induced distance is equal to the orbit space distance in $\Pi^m\SSS^{k-1}/\cO(k)$. 

To start with, let us consider the following group action of $\cO(k)$ on $\Pi_k^m\SSS^{k-1}$:
\#\label{equ:action-2}
(X,O)\in\Pi_k^m\SSS^{k-1}\times\cO(k)\to XO\in\Pi_k^m\SSS^{k-1},
\#
which is the group action \eqref{equ:action} restricted to the submanifold $\Pi_k^m\SSS^{k-1}$. Proposition~\ref{prop:5} shows that the group action \eqref{equ:action-2} is smooth, proper, free, and isometric.

\begin{proposition}\label{prop:5}
    The group action \eqref{equ:action-2} is smooth, proper, free, and isometric. 
\end{proposition}

\begin{proof}
    This is the group action \eqref{equ:action} restricted to the submanifold $\Pi_k^m\SSS^{k-1}$. Thus, by Proposition \ref{prop:4}, it is a smooth, proper, and isometric group action. Furthermore, the group action \eqref{equ:action-2} is free because for any $X\in\Pi_k^m\SSS^{k-1}$ and $O_1,O_2\in\cO(k)$, the equality $XO_1=XO_2$ implies that $O_1=O_2$. This uses the fact that $X\in\Pi_k^m\SSS^{k-1}$ is full-rank (i.e., rank-$k$). 
\end{proof}

Consequently, by Section \ref{sec:2.3.1}, the quotient set $\Pi^m_k\SSS^{k-1}/\cO(k)$ admits a natural Riemannian structure such that the quotient map $\pi_{k,k}$, defined in \eqref{equ:pi_kl}, is a Riemannian submersion. We denote by $g^{\qt,k}$ the Riemannian metric on this manifold. Since $\Pi_k^m\SSS^{k-1}/\cO(k)$ is diffeomorphic to $\Corr(m,k)$ via the map $\Phi_{k,k}$, as demonstrated by Proposition \ref{prop:smooth}, we can endow $\Corr(m,k)$ with a Riemannian metric $g^k$ such that the map $\Phi_{k,k}$ is a Riemannian isomorphism. Since $\Pi_k^m\SSS^{k-1}$ is incomplete, both quotient manifolds  $(\Pi^m_k\SSS^{k-1},g^{\qt,k})$ and $(\Corr(m,k),g^k)$ are not complete. Detailed geometric properties of them will be provided in Section \ref{sec:5}. 

To conclude this subsection, we show in Proposition \ref{prop:Riemannian} that the induced geodesic distance $d^{k}$ on $(\Corr(m,k),g^k)$ equals to the orbit space distance $d^{\orb,k}$ in $\Corr(m,[k])$. In other words, the metric structure of the orbit space $\Corr(m,[k])$ is compatible with the Riemannian structure of its principal stratum $\Corr(m,k)$.

\begin{proposition}\label{prop:Riemannian}
    Let $d^k$ be the geodesic distance in the Riemannian manifold $(\Corr(m,k),g^k)$. Let $d^{\orb,k}$ be the orbit space distance in $\Corr(m,[k])$. Then $d^k$ equals to $d^{\orb,k}$ when restricted to the subset $\Corr(m,k)$. 
\end{proposition}

\begin{proof}
    Equivalently, we only need to show that the distance $d^{\qt,k}$ {induced} by the Riemannian manifold $(\Pi_k^m\SSS^{k-1},g^{\qt,k})$ is equal to the quotient distance $d^{\orb,k}$ in the orbit space $\Pi^m\SSS^{k-1}/\cO(k)$. By Section \ref{sec:2.3.1}, we know $d^{\qt,k}$ is the quotient distance of the geodesic distance in $\Pi^m_k\SSS^{k-1}$. By Section \ref{sec:3.1}, we know $d^{\orb,k}$ is the quotient distance of the geodesic distance in $\Pi^m\SSS^{k-1}$. Since the geodesic distance in $\Pi^m_k\SSS^{k-1}$ equals to the geodesic distance in $\Pi^m\SSS^{k-1}$, their quotient distances are the same. This proves that $d^{\qt,k}$ equals to $d^{\orb,k}$ when restricted to the same domain.  
\end{proof}

\subsection{Different metrics $d^{\orb,k}$ for varying $k$}\label{sec:3.4}

In previous subsections, we have studied the orbit space $(\Corr(m,[k]),d^{\orb,k})$ for a fixed $k$. We have shown that it has a Euclidean subspace topology and its stratum $\Corr(m,\ell)$ can be smoothly embedded into the Euclidean space $\RR^{m\times m}$. One may expect that when $k$ is varying, the distances $d^{\orb, k}$ remain unchanged. However,  this is not true. 

In Proposition \ref{prop:counterexample}, we show that when $k$ varies, the distance $d^{\orb,k}$ also changes\footnote{This may not be a serious problem, considering the principal stratum is an open and dense set in the orbit space, and its Riemannian structure is compatible with the orbit space.}.  
Combining with Proposition \ref{prop:Riemannian}, we know that the distance $d^{\orb,k}$ in the orbit space $\Corr(m,[k])$ is only compatible with the Riemannian structure of its principal stratum $\Corr(m,k)$. For $\ell<k$, the Riemannian structure of $\Corr(m,\ell)$ will induce a distance $d^{\orb,\ell}$, which is different from $d^{\orb,k}$ by Proposition \ref{prop:counterexample}. 

Proposition \ref{prop:counterexample} emphasizes the importance of choosing an appropriate $k$ initially. Given that $\Corr(m)=\Corr(m,[m])$, meaning all $m\times m$ correlation matrices possess a rank of at most $m$, the default recommendation is to use $k=m$. However, under certain conditions, like bounded ranks, alternative choices for $k$ may be considered.

\begin{proposition}\label{prop:counterexample}
    Let $d^{\orb,k}$ be the distance in the orbit space $\Corr(m,[k])$. For $k_1<k_2\leq m$, we have 
    \$
    d^{\orb,k_2}(Z_1,Z_2)\leq d^{\orb,k_1}(Z_1,Z_2),\quad \forall Z_1,Z_2\in\Corr(m,[k_1]).
    \$
    Moreover, $d^{\orb,k_1}$ and $d^{\orb,k_2}$ are different, even when constrained to $\Corr(m,[k_1])$.
\end{proposition}

\begin{proof}
    The proof is provided in the Appendix \ref{sec:b2}.
\end{proof}

\section{Metric geometry of $\Corr(m,[k])\simeq \Pi^m\SSS^{k-1}/\cO(k)$}\label{sec:4}

In this section, we study detailed geometric properties and corresponding computations about the orbit space $\Corr(m,[k])\simeq\Pi^m\SSS^{k-1}/\cO(k)$. Before delving into our contents, let us briefly discuss \cite{journee2010low,boumal2015riemannian}, which also examined general optimization problems over the set of correlation matrices. In \cite{journee2010low}, the authors adopted the matrix factorization viewpoint and proposed a Riemannian trust-region method for optimization. In \cite{boumal2015riemannian}, the author studied the convex relaxation of the optimization problem over low-rank correlation matrices, and proposed the Riemannian staircase algorithm that takes optimization on low-rank correlation matrices as a building block. Unlike \cite{journee2010low,boumal2015riemannian}, which consider general optimization objectives, our discussions in this section only involve specific optimization objectives, for which simpler optimization algorithms can be designed. Specifically, in Section~\ref{sec:4.1}, we introduce a Riemannian gradient descent algorithm to compute the orbit space distance $d^{\orb,k}$. It essentially solves an alignment problem. In Section \ref{sec:4.2}, we establish some theories about the minimizing geodesics in the orbit space. We show that any minimizing geodesic has constant rank on the interior of the segment. Finally, in Section \ref{sec:4.3}, we address the problem of computing the weighted \Frechet mean in the orbit space. We develop an alternating minimization strategy for the computation. In the remaining part of this section, $m,k$ are held constant and $k\geq 2$.

\subsection{Distance}\label{sec:4.1}

We first examine the computation of the orbit space distance $d^{\orb,k}$ on $\Pi^m\SSS^{k-1}/\cO(k)$. In practice, this enables us to compute the pairwise distance matrix between a set of samples, cluster samples, and explore low-dimensional structures. 
Suppose $[X],[Y]$ are two points in the orbit space $\Pi^m\SSS^{k-1}/\cO(k)$. The squared distance between them is given by
\$
(d^{\orb,k}([X],[Y]))^2=\inf_{O\in\cO(k)}\ell_{X,Y}(O),
\$
where
\$
\ell_{X,Y}(O)\coloneqq\sum_{i=1}^m\arccos^2(\inner{(XO)_{i\cdot}}{Y_{i\cdot}}).
\$
This minimization problem does not admit a closed-form solution due to the nonlinear function $\arccos(\cdot)$. However, by employing Riemannian optimization techniques, we can efficiently compute this distance numerically. Specifically, we propose a Riemannian gradient descent method for computing the distance $d^{\orb,k}$. Fix $X,Y\in\Pi^m\SSS^{k-1}$. Initialize $O_0=\argmin_{O\in\cO(k)}\norm{XO-Y}_{\rm F}^2$, which has a closed-form solution given in  \cite{thanwerdas2023bures,massart2020quotient}. Then we compute the Riemannian gradient of $\ell_{X,Y}$ at $O$ as follows:
\$
\grad^{\cO(k)} \ell_{X,Y}(O)=\cP_{O}\left(\grad^E\ell_{X,Y}(O)\right),
\$
where $\cP_{O}(V)=O\ \skew(O^\top V)$ is to project $V\in\RR^{m\times m}$ to the tangent space of $\cO(m)$ at $O$, $\skew(W)=\frac{1}{2}(W-W^\top)$ is the skew-symmetric transformation, and $\grad^E(\cdot)$ is the standard Euclidean gradient in $\RR^{m\times m}$. In particular, we have
\#\label{equ:grad:dist}
\grad^{\cO(k)}\ell_{X,Y}(O)=O\ \skew\left(O^\top \sum_{i=1}^m\frac{-\arccos(\inner{(XO)_{i\cdot}{Y_{i\cdot}}}{})}{\sqrt{1-(\inner{(XO)_{i\cdot}}{Y_{i\cdot}})^2}}\cdot X_{i\cdot}Y_{i\cdot}^\top \right),
\#
where $X_{i\cdot},Y_{i\cdot}\in\RR^{k}$ are column vectors. At the $t$-th step, we compute the Riemannian gradient $\xi_t=\grad^{\cO(k)}\ell_{X,Y}(O_t)$ and choose a step size $\alpha_t$ via the Armijo rule~\cite{lee2012introduction}. Then we update $O_{t+1}$ as follows:
\#\label{equ:retraction}
O_{t+1}=R_{O_t}(\alpha_t\xi_t)
\#
where $R_{O_t}(\xi)={\rm qf}(O_t+\xi)$ is the retraction operation and ${\rm qf}(\cdot)$ returns the orthogonal matrix in the QR factorization. This procedure is repeated $T$ times and we return $\ell_{X,Y}(O_{T})$ as the squared distance $(d^{\orb,k}([X],[Y]))^2.$ The algorithm is summarized in Algorithm \ref{alg:1} and its convergence property is provided in Proposition \ref{prop:alg:1}.

\begin{algorithm}[t]\label{alg:1}
\caption{Riemannian gradient descent for distance on $\Pi^m\SSS^{k-1}/\cO(k)$}
\begin{algorithmic}
\REQUIRE $[X],[Y]\in\Pi^m\SSS^{k-1}/\cO(k)$, $T$.
\ENSURE squared distance $(d^{\orb,k}([X],[Y]))^2$.
\STATE Fix $X,Y$. 
\STATE Initialize $O_0\in\cO(k)$.
\WHILE{$t< T$}
\STATE Compute the Riemannian gradient, $\xi_t=\grad^{\cO(k)}\ell_{X,Y}(O_t)$, by \eqref{equ:grad:dist}.
\STATE Choose a step size $\alpha_t$ via the Armijo rule.
\STATE Update $O_{t+1}=R_{O_t}(\alpha_t\xi_t)$.
\ENDWHILE
\RETURN $\ell_{X,Y}(O_T)$.
\end{algorithmic}
\end{algorithm}

\begin{proposition}\label{prop:alg:1}
    Let $\{O_t\}$ be an infinite sequence of iterates generated by Algorithm \ref{alg:1}. Then every accumulation point of $\{O_t\}$ is a critical point of $\ell_{X,Y}(\cdot)$. 
\end{proposition}
\begin{proof}
    This is a corollary of Theorem 4.3.1 in \cite{absil2008optimization}.
\end{proof}

\subsection{Geodesics}\label{sec:4.2}

In this section, we provide some properties regarding the minimizing geodesics in the orbit space $(\Pi^m\SSS^{k-1}/\cO(k),d^{\orb,k})$. First, by Section \ref{sec:2.3.1}, we know this orbit space is a complete metric space and a length space. In particular, any pair of points in the space can be connected by a minimizing geodesic. Next, we show that any minimizing geodesic has constant rank on the interior of the segment in Proposition \ref{prop:geodesic}.

\begin{proposition}\label{prop:geodesic}  
    Let $\gamma:[0,1]\to\Pi^m\SSS^{k-1}/\cO(k)$ be a minimizing geodesic from $[X]$ to $[Y]$. If $\gamma_t=[X_t]$, then $\rank(X_t)$ equals to a constant $r$ for all $t\in(0,1)$, and \$
    r\geq \max\{\rank(X),\rank(Y)\}.
    \$
\end{proposition}

\begin{proof}
    Let $r=\max_{t\in[0,1]}\rank(X_t)$ and  $t_0\in[0,1]$ such that $\rank(X_{t_0})=r$. By Proposition \ref{prop:3}, we know for all $t\in(0,t_0)\cup(t_0,1)$, 
    \$
    (H_{k,\rank(X_t)})\leq (H_{k,\rank(X_{t_0})}),
    \$
    where $H_{k,\ell}$ is the stabilizer. By discussions in Section \ref{sec:3.2}, we know $(H_{k,\ell_1})\leq(H_{k,\ell_2})$ if and only if $\ell_1\geq \ell_2$. Therefore, $\rank(X_t)\geq r$ for all $t\in (0,t_0)\cup(t_0,1)$. By definition of $r$, we know $\rank(X_t)=r$ for all $t\in(0,1)$ and 
    \$
    r\geq \max\{\rank(X_0),\rank(X_1)\}=\max\{\rank(X),\rank(Y)\},
    \$
    which concludes the proof.
\end{proof}

\subsection{\Frechet mean}\label{sec:4.3}

In practice, given a set of points $\{x^i\}_{i=1}^n$ in a metric space $(\cM,d)$, we are often interested in computing the \Frechet mean of these points, which is given by\footnote{When there exist multiple minimizers, our goal is to find any one of them.} 
\$
\bar x=\argmin_{x\in\cM}\sum_{i=1}^nd^2(x^i,x).
\$
This generalizes the Euclidean mean to a metric space. In some applications such as \Frechet regression~\cite{petersen2019frechet}, we also show interests in computing a weighted version of the \Frechet mean. Given a vector of weights $w=(w_i)_{i=1}^n$, the weighted \Frechet mean of $\{x^i\}_{i=1}^n$ is given by
\#\label{equ:Frechetmean}
\bar x=\argmin_{x\in\cM}\sum_{i=1}^nw_i\cdot d^2(x^i,x).
\#
Unlike the Euclidean case, the existence and uniqueness of the \Frechet mean are not trivial. There are some works studying these questions in certain class of manifolds~\cite{karcher1977riemannian,kendall1990probability}. In addition, when the metric space $\cM$ is compact, the distance $d(\cdot,\cdot)$ is continuous, and the weights $w_i\geq0$, it can be shown that the weighted \Frechet mean exists. The uniqueness of the  \Frechet mean is more subtle, the study of which is beyond the scope of the current paper.

In this section, we will explore the \Frechet mean of $\{[X^i]\}_{i=1}^n$ in the orbit space $(\Pi^m\SSS^{k-1}/\cO(k),d^{\orb,k})$. Since this metric space is compact and the distance is continuous, the weighted \Frechet mean exists when the weights $w_i\geq 0$. However, computing the weighted \Frechet mean is challenging. In the sequel, we will present an alternating minimization algorithm that efficiently solves this computation problem.

\begin{algorithm}[t]
\caption{Alternating minimization for the \Frechet mean in $\Pi^m\SSS^{k-1}/\cO(k)$}\label{alg:2}
\begin{algorithmic}
\REQUIRE $\{[X^i]\}_{i=1}^n\subseteq\Pi^m\SSS^{k-1}/\cO(k)$, $w_i> 0$, and $S$.
\ENSURE the weighted \Frechet mean of $\{[X^i]\}$.
\STATE Fix $\{X^i\}$.
\STATE Initialize $X_0\in\Pi^m\SSS^{k-1}$ and $\{O^i_0\} \subseteq\cO(k)$.
\WHILE{$s<S$}
\STATE Fix $X_s$, and update $O^i_{s+1}$ by solving  \eqref{equ:O} for all $i$.
\STATE Fix $\{O^i_{s+1}\}$, and update $X_{s+1}$ by solving  \eqref{equ:X}.
\ENDWHILE
\RETURN $[X_S]$.
\end{algorithmic}
\end{algorithm}

Our algorithm is based on the following observation. Given $\{[X^i]\}_{i=1}^n$ in the orbit space $\Pi^m\SSS^{k-1}/\cO(k)$ and weights $w_i>0$, the weighted \Frechet mean is given by $[\bar X]$, where 
\$
\bar X&=\argmin_{X\in\Pi^m\SSS^{k-1}}\sum_{i=1}^n w_i\cdot (d^{\orb,k}([X^i],[X]))^2\\
&\overset{(i)}{=}\argmin_{X\in\Pi^m\SSS^{k-1}}\sum_{i=1}^n w_i\cdot \inf_{O^i\in\cO(k)}(d^{\Pi\SSS,k}(X^i O^i,X))^2\\
&\overset{(ii)}{=}\argmin_{X\in\Pi^m\SSS^{k-1}}\inf_{\{O^i\}\subseteq\cO(k)}\sum_{i=1}^n w_i\cdot (d^{\Pi\SSS,k}(X^i O^i,X))^2.
\$
Here equality (i) holds due to \eqref{equ:dist} and equality (ii) holds because $w_i>0$. As a result, to find the weighted \Frechet mean $[\bar X]$, it suffices to solve the following problem:
\#\label{equ:Frechet_mean}
\bar X,\{\bar O^i\}=\argmin_{X\in\Pi^m\SSS^{k-1},\{O^i\}\subseteq \cO(k)}\cL^{\rm mean}(X,\{O_i\}),
\#
where
\$
\cL^{\rm mean}(X,\{O_i\})\coloneqq \sum_{i=1}^n w_i\cdot (d^{\Pi\SSS,k}(X^i O^i,X))^2.
\$
We propose an alternating minimization strategy to solve this optimization problem. Specifically, we will alternatively update $X$ and $\{O^i\}$ with the other fixed. Initialize $X_0\in\Pi^m\SSS^{k-1}$ and $\{O^i_0\}\subseteq\cO(k)$. At the $s$-th step, we will first update $\{O^i_{s+1}\}$ when fixing $X_s$. This problem is separable, and for each $i$, we solve
\#\label{equ:O}
O^i_{s+1}=\argmin_{O^i\in\cO(k)}w_i\cdot (d^{\Pi\SSS,k}(X^i O^i,X_s))^2.
\#
This is the distance computation problem which has been addressed by Algorithm~\ref{alg:1}. Then, we fix $\{O^i_{s+1}\}$ and update $X_{s+1}$ by solving the following optimization problem:
\#\label{equ:X}
X_{s+1}=\argmin_{X\in\Pi^m\SSS^{k-1}}\cL^{\rm mean}(X,\{ O^i_{s+1}\}).
\#
This optimization problem can be regarded as computing the weighted \Frechet mean of $\{X^iO^i_{s+1}\}$ with weights $\{w_i\}$ in the spherical product manifold $\Pi^m\SSS^{k-1}$. We can solve it by Riemannian gradient descent, whose details are postponed to Section~\ref{sec:4.3.1}. By alternatively updating $X_s$ and $\{O^i_s\}$ upon convergence, we obtain the main algorithm for computing the weighted \Frechet mean of $\{[X^i]\}$ in the orbit space. This is summarized in Algorithm \ref{alg:2} and its property is established in Proposition \ref{prop:alg:2}. The convergence of the iterates $\{X_s,\{O^i_s\}\}$ is more challenging, which we leave to future work.

\begin{proposition}\label{prop:alg:2}
    Let $\{X_s,\{O^i_s\}\}$ be an infinite sequence of iterates generated by Algorithm \ref{alg:2}. Then the loss function decreases:
    \$
    \cL^{\rm mean}(X_{s+1},\{O^i_{s+1}\})\leq \cL^{\rm mean}(X_s,\{O^i_{s+1}\}),\quad\forall s.
    \$
    In particular, the loss function $\cL^{\rm mean}(X_{s},\{O^i_s\})$ converges as $s$ tends to $\infty$.  
\end{proposition}

\begin{proof}
    By definition, we know 
    \$
    \cL^{\rm mean}(X_s,\{O_{s+1}^i\})&\leq \cL^{\rm mean}(X_s,\{O_{s}^i\}), \\
    \cL^{\rm mean}(X_{s+1},\{O_{s+1}^i\})&\leq \cL^{\rm mean}(X_s,\{O_{s+1}^i\}).
    \$
    Therefore, $\cL^{\rm mean}(X_{s+1},\{O_{s+1}^i\})\leq\cL^{\rm mean}(X_s,\{O_{s}^i\})$ for all $s$. Since the loss function is non-negative, it converges as $s$ tends to $\infty$.
\end{proof}

\subsubsection{Riemannian gradient descent for problem \eqref{equ:X}}\label{sec:4.3.1}

\begin{algorithm}[t]
\caption{Riemannian gradient descent for solving problem \eqref{equ:X_j}}\label{alg-3}
\begin{algorithmic}
\REQUIRE $\{X^i\}_{i=1}^n\subseteq\Pi^m\SSS^{k-1}$, $w_i>0$, $\{O^i_{s+1}\}_{i=1}^n\subseteq\cO(k)$, and $T$. 
\ENSURE the minimizer of problem \eqref{equ:X_j}.
\STATE Initialize $X_{j\cdot,0}\in\SSS^{k-1}$.
\WHILE{$t<T$}
\STATE Compute the Riemannian gradient, $\xi_t=\grad^{\SSS}\ell^{\SSS,sj}(X_{j\cdot,t})$, by \eqref{equ:grad-S}.
\STATE Compute a step size $\alpha_t$ via the Armijo rule.
\STATE Update $X_{j\cdot,t+1}$ by \eqref{equ:retraction-S}.  
\ENDWHILE
\RETURN $X_{j\cdot,T}$.
\end{algorithmic}
\end{algorithm}

Now we elaborate how to solve problem \eqref{equ:X} using Riemannian gradient descent. First, we observe that problem \eqref{equ:X} is a separable problem. Substituting the definition of $d^{\Pi\SSS,k}$ into \eqref{equ:X}, we obtain
\$
X_{s+1}=\argmin_{X\in\Pi^m\SSS^{k-1}}\sum_{j=1}^m\sum_{i=1}^nw_i\cdot \arccos^2(\inner{(X^iO^i_{s+1})_{j\cdot}}{X_{j\cdot}}).
\$
Observe that the constraints on $\{X_{j\cdot}\}_{j=1}^m\subseteq\SSS^{k-1}$ are independent. Thus, addressing problem \eqref{equ:X} equates to independently solving the following problem for each $j\leq m$:
\#\label{equ:X_j}
(X_{s+1})_{j\cdot}=\argmin_{X_{j\cdot}\in\SSS^{k-1}}\ell^{\SSS,sj}(X_{j\cdot}),
\#
where 
\$
\ell^{\SSS,sj}(X_{j\cdot})\coloneqq  \sum_{i=1}^nw_i\cdot \arccos^2(\inner{(X^iO^i_{s+1})_{j\cdot}}{X_{j\cdot}}).
\$
This is an optimization problem on $\SSS^{k-1}$, and we employ Riemannian gradient descent to address this problem. For each $j$, we initialize $X_{j\cdot,0}\in\SSS^{k-1}$. At the $t$-th step, we compute the Riemannian gradient of $\ell^{\SSS,sj}(\cdot)$ as follows:
\#\label{equ:grad-S}
\grad^{\SSS}\ell^{\SSS,sj}(X_{j\cdot,t})=(I_k-X_{j\cdot,t} X_{j\cdot,t}^\top)\cdot \grad^{\rm E}\ell^{\SSS,sj}(X_{j\cdot,t})
\#
where $\grad^{\rm E}\ell^{\SSS,sj}(X_{j\cdot,t})$ is the gradient of $\ell^{\SSS,sj}(\cdot)$ at $X_{j\cdot,t}$ in the Euclidean space $\RR^{k}$:
\$
\grad^{\rm E}\ell^{\SSS,sj}(X_{j\cdot,t})=\sum_{i=1}^mw_i\cdot \frac{-\arccos(\inner{(X^iO^i_{s+1})_{j\cdot}}{X_{j\cdot,t}})}{\sqrt{1-(\inner{(X^iO^i_{s+1})_{j\cdot}}{X_{j\cdot,t}})^2}}\cdot (X^iO^i_{s+1})_{j\cdot}.
\$
Here $A_{j\cdot}$ for a matrix $A\in\RR^{m\times k}$ is treated as a column vector. Then we select a step size $\alpha_t$ by the Armijo rule. Next, we update $X_{j\cdot,t+1}$ 
\#\label{equ:retraction-S}
X_{j\cdot,t+1}=R^{\SSS}_{X_{j\cdot,t}}(\alpha_t\xi_t)\coloneqq\frac{X_{j\cdot,t}+\alpha_s\xi_t}{\norm{X_{j\cdot,t}+\alpha_t\xi_t}},
\#
where $\xi_t=\grad^{\SSS}\ell^{\SSS,sj}(X_{j\cdot,t})$. Repeat the iterations for $T$ times and  return $X_{j\cdot,T}$. This algorithm is summarized in Algorithm \ref{alg-3}, and its convergence property is provided in Proposition \ref{prop:alg-3}.

\begin{proposition}\label{prop:alg-3}
    Let $\{X_{j\cdot,t}\}$ be an infinite sequence generated by Algorithm \ref{alg-3}. Then any accumulating point of $\{X_{j\cdot,t}\}$ is a critical point of the loss function $\ell^{\SSS,sj}(\cdot)$. 
\end{proposition}

\begin{proof}
    This is a corollary of Theorem 4.3.1 in \cite{absil2008optimization}.
\end{proof}



\section{Riemannian geometry of $\Corr(m,k)\equiv \Pi_k^m\SSS^{k-1}/\cO(k)$}\label{sec:5}

In practical scenarios, working with correlation matrices with a fixed rank may be intriguing. For example, in the brain connectivity study in Section \ref{sec:app}, the correlation matrices from one site are of the same rank. Thus, it is natural to make the fixed rank assumption when building the model. In this section, we  delve into the detailed Riemannian geometric properties of the Riemannian quotient manifold $\Corr(m,k)$ of correlation matrices of fixed rank $k$. In Section~\ref{sec:3.3}, we have demonstrated that $\Corr(m,k)\equiv \Pi_k^m\SSS^{k-1}/\cO(k)$, where $\Phi_{kk}$ in \eqref{equ:Phi_kl} gives an isomorphism between them. In this section, we will examine the following quantities regarding this quotient manifold: horizontal and vertical spaces, Riemannian metric, injectivity radius, exponential and  logarithmic map, curvatures, gradient, connection, and Hessian. Theories are formulated for $\Pi^m_k\SSS^{k-1}/\cO(k)$, and the mapping  $\Phi_{kk}$ can be utilized to extrapolate corresponding theories onto $\Corr(m,k)$. Throughout this section, $k,m$ are held constant and $k\geq 2$.

To begin with, we briefly review the manifold $\Pi^m_k\SSS^{k-1}$ endowed with the spherical product metric $g^{\Pi\SSS,k}$. Given $X\in\Pi^m_k\SSS^{k-1}$, the tangent space to $\Pi^m_k\SSS^{k-1}$ at $X$ is given by
\$
\cT_{X}\Pi^m_k\SSS^{k-1}=\{V\in\RR^{m\times k}\mid \diag(V X^\top)=\zero\}.
\$
The spherical product metric $g^{\Pi\SSS,k}$ is defined by
\#\label{equ:metric-S}
g^{\Pi\SSS,k}_{X}(V,W)=\tr(V^\top W),\quad \forall V,W\in \cT_{X}\Pi^m_k\SSS^{k-1}.
\#
The induced geodesic distance $d^{\Pi\SSS,k}$ is given by \eqref{equ:spherical-dist}. In Section \ref{sec:3.3}, we have shown that the group action \eqref{equ:action-2} of $\cO(k)$ on the Riemannian manifold $(\Pi^m_k\SSS^{k-1},d^{\Pi\SSS,k})$ is smooth, proper, free, and isometric. Thus, the quotient space $\Pi^m_k\SSS^{k-1}/\cO(k)$ admits a Riemannian manifold structure such that the quotient map $\pi_{k,k}$, defined in \eqref{equ:pi_kl}, is a Riemannian submersion.

\subsection{Horizontal and vertical spaces}

The tangent space $\cT_{X}\Pi^m_k\SSS^{k-1}$ can be represented as the direct sum of the vertical and horizontal spaces at $X$, defined by
\$
\cV_X&=\{X\Omega\mid\Omega\in\cS_{\skew}(k)\},\\
\cH_X=\cV_X^\perp&=\{V\in\RR^{m\times k}\mid V^\top X=X^\top V,\diag(VX^\top)=\zero \}.
\$
The vertical space $\cV_X$ is the tangent space to the orbit $[X]=\{XO\mid O\in\cO(k)\}$ at $X$, and the horizontal space is its orthogonal complement. Given any $W\in\cT_X\Pi^m_k\SSS^{k-1}$, the projection of $W$ on the horizontal and vertical spaces, denoted by $P^v_{X}(W)$ and $P^h_X(W)$, are derived in \cite{massart2020quotient} as follows:
\#
P^v_X(W)&=X{\bf T}^{-1}_{X^\top X}(X^\top W-W^\top X),\notag\\
P^h_X(W)&=W-P^v_X(W),\label{equ:h-proj}
\#
where ${\bf T}_{E}^{-1}(\cdot)$ is the inverse of the following map
\$
{\bf T}_{E}:\RR^{k\times k}\to \RR^{k\times k}, \quad A\to EA+AE.
\$
Lemma A.10 in \cite{massart2020quotient} states that when $E$ is SPD, ${\bf T}_{E}$ is invertible and ${\bf T}_{E}(\cS_{\skew}(k))=\cS_{\skew}(k)$.
The solution $A$ of the Sylvester equation $EA+AE=W$ can be efficiently computed by diagonalization of the SPD matrix $E$ (see \cite{massart2020quotient}, Section 2.2, \cite{bhatia1997and}, Section 10). 

For any tangent vector $\xi$ on $\Pi^m_k\SSS^{k-1}/\cO(k)$ at $[X]$, there exists a unique vector $\xi^{\sharp}_X\in\cH_X$ such that 
\$
d_{X}\pi_{k,k}(\xi^\sharp_X)=\xi, 
\$
where $d_{X}\pi_{k,k}(\cdot)$ is the differential of the quotient map $\pi_{k,k}$ at $X$. We call $\xi^\sharp_X$ the horizontal lift of $\xi$ at $X$ (see \cite{kobayashi1996foundations,absil2008optimization}). Many geometric quantities in a quotient manifold involve using the horizontal lifts.   

\subsection{Riemannian metric, geodesic distance, and injectivity radius}

The Riemannian manifold $(\Pi^m_k\SSS^{k-1},g^{\Pi\SSS,k})$ induces a Riemannian metric on the quotient manifold $\Pi^m_k\SSS^{k-1}/\cO(k)$ via the quotient map $\pi_{k,k}$. Specifically, the relation
\#\label{equ:metric-quotient}
g^{\qt,k}_{\pi_{k,k}(X)}(\xi,\zeta)=g^{\Pi\SSS,k}_{X}(\xi_X^\sharp,\zeta_X^\sharp),\quad\forall \xi,\zeta\in\cT_{\pi_{k,k}(X)}\Pi^m_k\SSS^{k-1}/\cO(k),
\#
defines a Riemannian metric on $\Pi^m_k\SSS^{k-1}/\cO(k)$.  $\pi_{k,k}$ is then a Riemannian submersion from $(\Pi^m_k\SSS^{k-1},g^{\Pi\SSS,k})$ to the quotient manifold $(\Pi^m_k\SSS^{k-1}/\cO(k),g^{\qt,k})$. Moreover, the Riemannian metric $g^{\qt.k}$  induces a geodesic distance that equates to the orbit distance $d^{\orb,k}$, which we prove in Proposition \ref{prop:Riemannian}. The metric space $(\Pi^m_k\SSS^{k-1}/\cO(k),d^{\orb,k})$ is incomplete and is dense in the orbit space $(\Pi^m\SSS^{k-1}/\cO(k),d^{\orb,k})$. In particular, it implies that the injectivity radius of the quotient manifold is zero.

\begin{corollary}\label{corr:inj_radius}
    The injectivity radius of the manifold $\Pi^m_k\SSS^{k-1}/\cO(k)$ is zero. 
\end{corollary}

\begin{proof}
    Recall that $\Pi^m_k\SSS^{k-1}/\cO(k)$ is an incomplete Riemannian manifold because it is a dense proper set in the orbit space $\Pi^m\SSS^{k-1}/\cO(k)$. The corollary follows because any incomplete Riemannian manifold has zero injectivity radius. 
\end{proof}

\subsection{Exponential map}

We first recall the exponential map on a sphere. Suppose $x\in\SSS^{k-1}$ and $v\in\cT_x\SSS^{k-1}$. The exponential map $\Exp_x^{\SSS}(v)$ is given by 
\$
\Exp_x^{\SSS}(v)=x\cos(\norm{v})+\sin(\norm{v})\frac{v}{\norm{v}}.
\$
Then we can derive the exponential map on the spherical product manifold $\Pi^m\SSS^{k-1}$. Let $X\in\Pi^m\SSS^{k-1}$ and $V\in\cT_{X}\Pi^m\SSS^{k-1}$. Then the exponential map $\Exp^{\Pi\SSS}_X(V)$ on this manifold is given by $Y\in\Pi^m\SSS^{k-1}$ such that
\$
Y_{i\cdot}=\Exp^{\SSS}_{X_{i\cdot}}(V_{i\cdot}),\quad\forall i\leq m.
\$
When we restrict ourselves to the open submanifold $\Pi^m_k\SSS^{k-1}$, we get the exactly same expression of the exponential map, as long as it is well-defined. Let $X\in\Pi^m_k\SSS^{k-1}$ and $V\in\cT_{X}\Pi^m_k\SSS^{k-1}$. The definition interval $I_{X,V}$ of the exponential map is not necessarily $\RR$, as $\Pi^m_k\SSS^{k-1}$ is not complete. In particular, the definition interval has the following characterization:
\$
I_{X,V}=(t^{\min}_{X,V},t^{\max}_{X,V}),
\$
where
\$
t^{\min}_{X,V}&=\min\{t<0\mid \rank(\Exp_X^{\Pi\SSS}(s V))=k,\forall s\in(t,0)\},\\
t^{\max}_{X,V}&=\max\{t>0\mid \rank(\Exp_X^{\Pi\SSS}(s V))=k,\forall s\in(0,t)\}.
\$
Now we examine the quotient manifold $\Pi^m_k\SSS^{k-1}/\cO(k)$. Let $[X]\in\Pi^m_k\SSS^{k-1}/\cO(k)$ and $\xi\in\cT_{X}\Pi^m_k\SSS^{k-1}/\cO(k)$. Proposition \ref{prop:quotient-geodesic} implies that the definition interval $I_{[X],\xi}$ of the exponential map on the quotient manifold contains $I_{X,\xi^\sharp_X}$ as a sub-interval, where $\xi^\sharp_X$ is the horizontal lift of $\xi$ at $X$. Also, the exponential map on the quotient manifold satisfies
\$
\Exp^{\qt}_{[X]}(t\xi)=[Y_t],\quad \textnormal{where }Y_t=\Exp^{\Pi\SSS}_{X}(t\xi_X^\sharp),
\$
for all $t\in I_{X,\xi_X^\sharp}$. 

\subsection{Logarithmic map}

In this section, we study the logarithmic map. To begin with, we establish the following result. It shows that minimizing geodesics in the orbit space and in the Riemannian quotient manifold are highly relevant.  

\begin{proposition}\label{prop:5.2}
    Let $[X],[Y]\in\Pi^m_k\SSS^{k-1}/\cO(k)$. Then any minimizing geodesic connecting $[X],[Y]$ in the orbit space $(\Pi^m\SSS^{k-1}/\cO(k),d^{\orb,k})$ is a minimizing geodesic connecting $[X],[Y]$ in the Riemannian quotient manifold $(\Pi^m_k\SSS^{k-1}/\cO(k),g^{\qt,k})$, and vice versa.  
\end{proposition}

\begin{proof}
    Proposition \ref{prop:geodesic} implies that any minimizing geodesic connecting $[X],[Y]$ has constant rank $k$ on the interior of the segment. Therefore, the whole minimizing geodesic is located in the space $\Pi^m_k\SSS^{k-1}/\cO(k)$. Proposition \ref{prop:Riemannian} demonstrates that the Riemannian quotient manifold induced the orbit space distance. Therefore, any minimizing geodesic connecting $[X],[Y]$ in the orbit space is also a minimizing geodesic connecting $[X],[Y]$ in the Riemannian quotient manifold. The reverse of the statement is also true due to the compatibility of the metric.  
\end{proof}

In particular, for any $[X],[Y]\in\Pi^m_k\SSS^{k-1}/\cO(k)$, there exists a minimizing geodesic  connecting them\footnote{The uniqueness is not guaranteed.}. In the following, we will give an approach to compute a logarithm of $[Y]$ from $[X]$. The idea is to utilize Proposition \ref{prop:1}. First, we address the following alignment problem:
\$
\cO^*=\argmin_{O\in\cO(k)}d^{\Pi \SSS,k}(X,YO).
\$
This problem is addressed in the computation of distance in Section \ref{sec:4.1}. Specifically,  Algorithm \ref{alg:1} can be used to efficiently compute $O^*$. Then the pair of points $X,YO^*$ are registered. Let $\varphi(t)$ be a constant-speed minimizing geodesic in $\Pi^m\SSS^{k-1}$ such that $\varphi(0)=X$ and $\varphi(1)=YO^*$. By Proposition \ref{prop:1} and Proposition \ref{prop:5.2}, we know $\gamma(t)=[\varphi(t)]$ is a constant-speed minimizing geodesic in the orbit space $(\Pi^m\SSS^{k-1},d^{\orb,k})$ and the Riemannian quotient manifold $(\Pi^m_k\SSS^{k-1},g^{\qt,k})$ such that  $\gamma(0)=[X]$, $\gamma(1)=[Y]$. Consequently, a logarithm of $[Y]$ from $[X]$ is given by $\gamma'(0)$. Finally, it is worth noting that we do not address the issue of the uniqueness of the logarithm. This is a more challenging problem, which we leave to future study.

\subsection{Curvature}

In this section, we aim to demonstrate that the quotient manifold $\Pi^m_k\SSS^{k-1}/\cO(k)$ has non-negative sectional curvatures. To this end, we first study the sectional curvatures of the spherical product manifold $\Pi^m\SSS^{k-1}$. In specific, there are two cases:
\vspace{5pt}
    {
    \renewcommand{\labelitemi}{\scalebox{0.4}{$\bullet$}}
    
\begin{itemize}
    \item When $k=2$, $\Pi^m\SSS^{k-1}$ has zero sectional curvatures.

    \vspace{5pt}
    \item When $k>2$, by  Chapter 3, Section 1 and Section 2.2 in \cite{petersen2006riemannian}, we find that the sectional curvatures of $\Pi^m\SSS^{k-1}$ are restricted to the interval $[0,1]$. 
\end{itemize}
}
\vspace{5pt}
\noindent Therefore, $\Pi^m\SSS^{k-1}$, and thus $\Pi^m_k\SSS^{k-1}$ has non-negative sectional curvatures. By the O'Neill's submersion result \cite{o1966fundamental}, we know the curvatures of the quotient manifold are larger than those of the original manifold. In particular, the sectional curvatures of the quotient manifold $\Pi^m_k\SSS^{k-1}/\cO(k)$ are larger than the minimum sectional curvatures of $\Pi^m_k\SSS^{k-1}$. Hence, the quotient manifold also has non-negative sectional curvatures.

\subsection{Gradient, connection, and Hessian}

This section delves into the Riemannian gradient, connection, and Hessian on the quotient manifold $\Pi^m_k\SSS^{k-1}/\cO(k)$. Specifically, let $f:\Pi^m_k\SSS^{k-1}/\cO(k)\to\RR$, and $f^{\sharp}:\Pi^m_k\SSS^{k-1}\to\RR$ be functions such that $f^\sharp=f\circ \pi_{k,k}$, where $\pi_{k,k}$ is the quotient map defined in \eqref{equ:pi_kl}. By Section 3.6.2 in \cite{absil2008optimization}, the Riemannian gradient of $f$ satisfies 
\$
(\grad f([X]))^{\sharp}_X=\grad f^\sharp(X),\quad\forall X\in\Pi^m_k\SSS^{k-1},
\$
where $\grad f^\sharp$ is the Riemannian gradient of $f^\sharp$ on $\Pi^m_k\SSS^{k-1}$. By Section 5.3.4 in \cite{absil2008optimization}, the Riemannian connection $\nabla$ on the quotient manifold $\Pi^m_k\SSS^{k-1}/\cO(k)$ satisfies 
\$
(\nabla_\eta\xi)^\sharp_X=P^h_X((\nabla^\sharp_{\eta^\sharp}\xi^\sharp)_X),\quad \forall X\in\Pi^m_k\SSS^{k-1},
\$
for all vector fields $\eta,\xi$ on the quotient manifold, where $\eta^\sharp,\xi^\sharp$ are the horizontal lifts of $\eta,\xi$, $P^h_X$ is the horizontal projection in \eqref{equ:h-proj}, and $\nabla^\sharp$ is the Riemannian connection on $\Pi^m_k\SSS^{k-1}$. By Section 5 in \cite{absil2008optimization} and Proposition A.14 in \cite{massart2020quotient}, the Riemannian Hessian of $f$ on the quotient manifold $\Pi^m_k\SSS^{k-1}/\cO(k)$ satisfies
\$
(\Hess f([X])[\xi_{[X]}])^\sharp_X=P^h_X((\nabla^\sharp_{\xi^\sharp}\grad f^\sharp)_X),\quad \forall X\in\Pi^m_k\SSS^{k-1},
\$
where $\xi$ is any vector field on the quotient manifold.


\section{Application}\label{sec:app}

This section presents an application of our geometric approach to analyze brain connectivity patterns in autism spectrum disorder (ASD). ASD is a neurodevelopmental disorder characterized by abnormal social skills, communication skills, interests and behavior patterns~\cite{liu2020improved,hirota2023autism,lord2018autism}. A previous survey \cite{baio2018prevalence} observed the prevalence of ASD among children and showed that about 1 in 59 children worldwide has ASD. Thus, ASD has significantly affected the society and  caused great economic burden for both patients' families and society \cite{rogge2019economic}. Therefore, it is crucial to develop effective diagnostic and treatment methods.

Among available technologies, resting-state functional magnetic resonance imaging (rs-fMRI) demonstrates potential for defining functional neurophenotypes \cite{kelly2008competition,van2012future}. Many studies have utilized brain rs-fMRI data to develop classifiers to distinguish patients with ASD from typically developing controls (TDC) \cite{abraham2017deriving,liu2020improved}. In their methods, brain rs-FMRI data is typically transformed into correlation matrices to characterize brain connectivity patterns. This motivates us to use the proposed quotient geometry of correlation matrices for such brain connectivity study. 

In our study, we collect data from the Autism Brain Imaging Data Exchange I (ABIDE I) \cite{di2014autism}. The ABIDE I contains 539 subjects with ASD and 573 TDCs from 17 international sites. In our study, we only use the data from a single site, Caltech, to reduce site-wise heterogeneity. Also, we use the preprocessed data provided by the the Preprocessed Connectome Project initiative \cite{craddock2013neuro}. Specifically, we choose the data processed with the Configurable Pipeline for the Analysis of Connectomes (CPAC), the procedures of which include slice time correction, motion correct, skull-strip, nuisance signal regression. Our study focuses on the time series data for regions of interest for the Harvard-Oxford Atlas. For each subject, we compute the covariance matrix of the time series and we delete the subject which has more than 8 zero diagonal entries in its covariance matrix. This leaves us 35 subjects with 18 ASDs and 17 TDCs. The average age of these subjects is 27.45 (with a std of 10.23), and the Male/Female ratio is 27/8. Next, according to these subjects, we select all regions of interest where the associated time series exhibit non-zero variance for all subjects. This process results in 104 regions of interest. 

In our study, we compute the correlation matrices for all subjects and these correlation matrices are of fixed rank 62. The consistent rank across all correlation matrices is a noteworthy observation, further justifying our use of the quotient geometry proposed in this paper. Next, we divide correlation matrices into two groups according to the ASD/TDC labels. For each group, we compute the \Frechet mean of  correlation matrices using the orbit space structure with rank parameter equal to 62. These two mean correlation matrices are visualized in the first row of Figure \ref{fig:2}. Then we compute the distance between these two mean correlation matrices, which is 3.69. Furthermore, we visualize the difference between these two mean correlation matrices in the second row of Figure \ref{fig:2}. After thresholding the difference matrix at 0.2, we find that the main differences between two mean correlation matrices are mostly concentrated in a subset of regions of interest. Thus, a further examination of these regions of interest from a biological perspective may be intriguing.  

\begin{figure}
    \centering
    \includegraphics[width=0.9\linewidth]{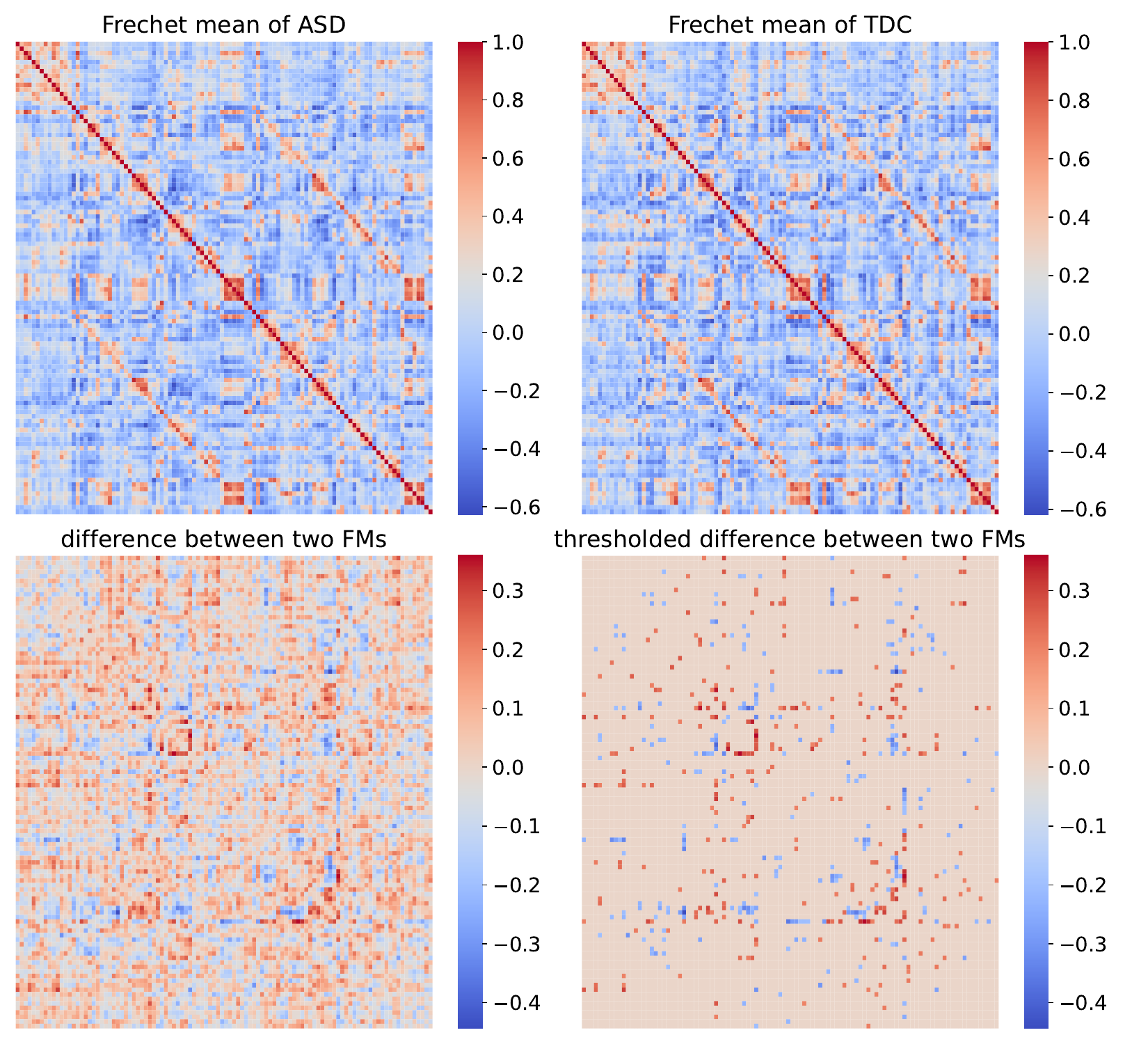}
    \caption{Heat maps of mean correlation matrices. In the first row, the left panel shows the mean correlation matrix for the ASD group, and the right panel shows the mean correlation matrix for the TDC group. In the second row, the left panel shows the difference of two mean correlation matrices (ASD - TDC), and the right panel shows the difference  after thresholding at 0.2, i.e., setting those values in [-0.2,0.2] to 0. }
    \label{fig:2}
\end{figure}


\section{Conclusion}\label{sec:conclusion}

This paper provides a comprehensive theory on the quotient geometry of correlation matrices. We demonstrate that the set $\Corr(m,[k])$ of bounded-rank correlation matrices admits a natural orbit space structure. Its stratification is determined by the rank of matrices. Moreover, its principal stratum $\Corr(m,k)$ has a compatible Riemannian quotient manifold structure. Both distance and \Frechet mean in the orbit space $\Corr(m,[k])\equiv\Pi^m\SSS^{k-1}/\cO(k)$ do not admit closed-form expression. Hence, our paper provides Riemannian optimization algorithms for efficient numerical computations of these two quantities. Moreover, we present theories on the minimizing geodesics in the orbit space, which are shown to have constant rank on the interior of the segment. Also, we delve into geometric quantities in the Riemannian quotient manifold $\Pi^m_k\SSS^{k-1}/\cO(k)$. We study concepts including horizontal and vertical spaces, Riemannian metric, injectivity radius, exponential and logarithmic map, curvatures, gradient and Hessian. Finally, we apply our geometric approach to an autism study and illustrate the utility of our method. To conclude, let us present several directions worthy of future study. 
\vspace{5pt}
    {
    \renewcommand{\labelitemi}{\scalebox{0.4}{$\bullet$}}
    \begin{itemize}
        \item First, it is intriguing to apply the proposed quotient geometry in various fields involving correlation matrices. {For example, one may conduct a large-scale brain connectivity study using this geometry. Besides conducting tasks such as regression and classification, one can also take nonstationarity into account and investigate dynamic functional connectivity \cite{liu2020improved}. It is interesting to use advanced statistical tools designed for stratified spaces, such as stratified PCA \cite{szwagier2023stratified} and central limit theorem for \Frechet mean on stratified spaces \cite{mattingly2023central}. }  
        
\vspace{5pt}

        \item Second, for general \Frechet regression~\cite{petersen2019frechet}, one may need to find the weighted \Frechet mean with some negative weights $w_i$. Section \ref{sec:4.3} only addresses the case where all $w_i>0$. For general $w_i$, it is interesting to develop a suitable algorithm.
     
\vspace{5pt}
        \item Third, it is interesting to study general optimization over the set of correlation matrices of bounded rank. It is of interest to examine whether the algorithm in \cite{olikier2023first} can be extended to this case.

\vspace{5pt}
        \item {Fourth, it is worth determining an expression for the horizontal life, the number of logarithms, and an expression for the curvature. }

\vspace{5pt}
        \item {Finally, since the product of spheres is symmetric, one can define Riemannian radial distributions on such space  and develop estimation theory \cite{chen2024riemannian}. Thus, it is interesting to explore whether one can define radial distributions on orbit spaces and develop similar statistical theory. } 

    \end{itemize}
}

\appendix 

\section{Geometry} 

In this section, we review more geometric tools.  

\subsection{Diffeomorphism}

A bijection $\Phi:\cM\to\cN$ between smooth manifolds $\cM$ and $\cN$ is a diffeomorphism when both $\Phi$ and $\Phi^{-1}$ are differentiable. Proposition \ref{prop:differentiable} provides a useful tool for establishing the differentiable property of a map.

\begin{proposition}[\cite{brickell1970differentiable}, Proposition 6.1.2]\label{prop:differentiable} Let $f$ be a submersion of a smooth manifold $\cM$ to a smooth manifold $\cN$. Let $g:\cN\to\cN'$ be a map from $\cN$ to a smooth manifold $\cN'$. If $g\circ f$ is differentiable, then $g$ is also differentiable.
\end{proposition}

\subsection{Submanifolds}

Let $h:\cM\to\cN$ be any map and $c$ be any point of $\cN$. We call the set $h^{-1}(c)$ a level set of $h$. Let $h$ be a smooth map between smooth manifolds $\cM$ and $\cN$. A point $x\in\cM$ is said to be a regular point of $h$ if the differential $d_{x}h:\cT_{x}\cM\to \cT_{h(x)}\cN$ is surjective. A level set $h^{-1}(c)$ is called a regular level set if every element in $h^{-1}(c)$ is a regular point of $h$. Proposition \ref{prop:levelset} shows that every regular level set is an embedded submanifold.

\begin{proposition}[\cite{lee2012introduction}, Corollary 5.14]\label{prop:levelset}
    Every regular level set of a smooth map between smooth manifolds is a properly embedded submanifold.
\end{proposition}

Suppose $\cS$ is a regular level set of a smooth map $h:\cM\to\RR^m$, and is treated as an embedded submanifold of $\cM$. Proposition \ref{prop:submanifold} characterizes the tangent space of $\cS$ at each point as a subspace of the tangent space of $\cM$ at the same point.  

\begin{proposition}[\cite{lee2012introduction}, Corollary 5.39]\label{prop:submanifold}
    Suppose $\cM$ is a smooth manifold and $\cS\subseteq \cM$ is an embedded submanifold. Suppose $\cS$ is a level set of a smooth submersion $h:\cM\to\RR^{m}$. For $p\in\cS$, a vector $v\in T_{p}\cM$ is tangent to $\cS$ if and only if $d_{p}h(v)=0$.
\end{proposition}

\subsection{Lie group actions}

To employ the quotient manifold theorem~\cite{lee2012introduction}, one needs to study the properties of Lie group actions. One useful tool is the following proposition.

\begin{proposition}[\cite{lee2012introduction}, Corollary 21.6]\label{prop:a2}
    Every continuous action by a compact Lie group on a manifold is proper. 
\end{proposition}

\section{Omitted proofs}

In this section, we provide the proofs that are omitted from the main text. For readers' convenience, we will restate the propositions. 

\subsection{Proof of Proposition \ref{prop:topology}}\label{sec:topology}

\begin{proposition}
    The metric spaces $(\Corr(m,[k]),d^{\orb,k})$ and  $(\Corr(m,[k]),d^{\rm E})$ share the same topology. 
\end{proposition}

\begin{proof}
    We prove this proposition in two steps.
    \vspace{5pt}
    {
    \renewcommand{\labelitemi}{\scalebox{0.4}{$\bullet$}}
    \begin{itemize}
    \item First, we claim that when
    \$
    X_nX_n^\top\to XX^\top\textnormal{ in }d^{\orb,k},\quad \textnormal{where }X_n,X\in\Pi^m\SSS^{k-1},
    \$
    it holds that $X_nX_n^\top\to XX^\top$ in $d^{\rm E}$.     Indeed, when $X_nX_n^\top\to XX^\top$ in $d^{\orb,k}$, it implies that there exists $\tilde X_n\in [X_n]$ such that $\tilde X_n\to X$ in $d^{\Pi\SSS,k}$. Note that $(\Pi^m\SSS^{k-1},d^{\Pi\SSS,k})$ has the same topology as $(\Pi^m\SSS^{k-1},d^{{\rm E},k})$, where $d^{{\rm E},k}$ is the Euclidean distance in $\RR^{m\times k}$. It implies that $\tilde X_n\to X$ in $d^{{\rm E},k}$. Since the map
    \$
    \phi_k:\Pi^m\SSS^{k-1}\to \RR^{m\times m}, \quad Y\to YY^\top,
    \$
    is continuous, we conclude that $\tilde X_n\tilde X_n^\top\to XX^\top$ in $d^{\rm E}$. The first claim then follows from the fact that $\tilde X_n\in[X_n]$ and $X_nX_n^\top=\tilde X_n\tilde X_n^\top$.

    \item Next, we claim that when 
    \$
    X_nX_n^\top\to XX^\top\textnormal{ in }d^{\rm E},\quad \textnormal{where }X_n,X\in\Pi^m\SSS^{k-1},
    \$
    it holds that $X_nX_n^\top\to XX^\top$ in $d^{\orb,k}$. Indeed, when $X_nX_n^\top\to XX^\top$ in $d^{\rm E}$, we have $X_nX_n^\top\to XX^\top$ in $d^{\rm BW}$ as $(\Cov(m),d^{\rm E})$ and $(\Cov(m),d^{\rm BW})$ have the same topology~\cite{thanwerdas2023bures}, where $d^{\rm BW}$ is the Bures-Wasserstein distance in $\Cov(m)$. Proposition 5.1 in \cite{massart2020quotient} implies that there exists $\tilde X_n\in[X_n]$ such that $\tilde X_n\to X$ in $d^{{\rm E},k}$. Thus, $\tilde X_n\to X$ in $d^{\Pi\SSS,k}$, as $(\Pi^m\SSS^{k-1},d^{{\rm E},k})$ and $(\Pi^m\SSS^{k-1},d^{\Pi\SSS,k})$ have the same topology. Therefore, $\tilde X_n\tilde X_n^\top\to XX^\top$ in $d^{\orb,k}$, and the second claim holds since $X_nX_n^\top=\tilde X_n\tilde X_n^\top$. 
    \end{itemize}
    }
    \vspace{5pt}
    Combining these two steps, we conclude the proof of this proposition.
\end{proof}

\subsection{Proof of Proposition \ref{prop:smooth}}\label{sec:b1}

\begin{proposition}
    $\Corr(m,\ell)$ is a smooth embedded submanifold of $\RR^{m\times m}$. The bijection $\Phi_{k,\ell}$ in \eqref{equ:Phi_kl} is a diffeomorphism between two smooth manifolds. 
\end{proposition}

\begin{proof}

    First, we establish the smooth structure of $\Corr(m,\ell)$ by utilizing its relationship with $\Cov(m,\ell)$, which has been shown to be a smooth embedded submanifold of $\RR^{m\times m}$~\cite{vandereycken2009embedded}. The tangent space of $\Cov(m,\ell)$ at $Z\in\Cov(m,\ell)$ is given by 
    \$
    \cT_{Z}\Cov(m,\ell)=\{\Delta Z+Z\Delta^\top\mid \Delta\in \RR^{m\times m}\}.
    \$
    Consider the smooth map 
    \#
    \diag_{m,\ell}:\Cov(m,\ell)\to \RR^{m}, \quad Z\to \diag(Z),
    \#
    where $\diag(Z)$ is a vector consisting of the diagonal elements of $Z$, defined in Section \ref{sec:notation}. Let $c$ be an $m$-dimensional vector whose entries are all 1. Then the set $\Corr(m,\ell)$ is a level set of $\diag_{m,\ell}$ in the sense that $\Corr(m,\ell)=\diag_{m,\ell}^{-1}(c)$. For $Z\in\Cov(m,\ell)$, the differential $d_{Z}(\diag_{m,\ell}):\cT_{Z}\Cov(m,\ell)\to \RR^m$ is given by
    \$
    V\to \diag(V),\quad V=\Delta Z+Z\Delta^\top\in \cT_{Z}\Cov(m,\ell), \quad \Delta\in\RR^{m\times m}.
    \$
    Therefore, the differential is surjective when $Z\in \Corr(m,\ell)$. Consequently, $\Corr(m,\ell)$ is a regular level set of $\diag_{m,\ell}$ and by Proposition \ref{prop:levelset}, $\Corr(m,\ell)$ is a smooth embedded submanifold of $\Cov(m,\ell)$. In particular, it is a smooth embedded submanifold of $\RR^{m\times m}$. By Proposition \ref{prop:submanifold}, the tangent space of $\Corr(m,\ell)$ at each point $Z$ is given by
    \#\label{equ:tangent_corr}
    \cT_{Z}\Corr(m,\ell)=\{\Delta Z+Z\Delta^\top\mid \Delta\in\RR^{m\times m},\diag(\Delta Z)=0\}.
    \#
    When $Z=YY^\top$ for $Y\in\Pi_\ell^m\SSS^{\ell-1}$, we can provide a concise characterization of this tangent space. Let $Y_{\perp}\in\St(m,m-\ell)$ be an orthonormal matrix such that $Y_{\perp}^\top Y=\zero$. Then we can rewrite $\Delta\in\RR^{m\times m}$ as follows:
    \$
    \Delta=\xi (Y^\top Y)^{-1}Y^\top+\zeta Y_{\perp}^\top,
    \$
    where $\xi\in\RR^{m\times \ell}$ and $\zeta\in\RR^{m\times (m-\ell)}$. Using this representation, we have
    \$
    \Delta Z+Z\Delta^\top=\xi Y^\top+Y \xi^\top.
    \$
    Substituting it into \eqref{equ:tangent_corr}, we can rewrite the tangent space in \eqref{equ:tangent_corr} as
    \#\label{equ:tangent_corr_concise}
    \cT_{Z}\Corr(m,\ell)=\{\xi Y^\top+Y\xi^\top\mid \xi\in\RR^{m\times \ell},\diag(\xi Y^\top)=\zero\}.
    \#
    We will use this characterization in the following proof.

    Now we show that the bijection $\Phi_{k,\ell}$ in \eqref{equ:Phi_kl} is a diffeomorphism between smooth manifolds $\Corr(m,\ell)$ and $\Pi_{\ell}^m\SSS^{k-1}/\cO(k)$. Define the map
    \#\label{equ:phi_kl}
    \phi_{k,\ell}:\Pi^m_\ell\SSS^{k-1}\to \Corr(m,\ell),\quad X\to XX^\top.
    \#
    It holds that $\phi_{k,\ell}=\Phi_{k,\ell}\circ \pi_{k,\ell}$ and $\pi_{k,\ell}=\Phi_{k,\ell}^{-1}\circ \phi_{k,\ell}$, where $\pi_{k,\ell}$ is given in \eqref{equ:pi_kl}. 
    By Proposition \ref{prop:differentiable}, if we show that $\pi_{k,\ell}$ and $\phi_{k,\ell}$ are both smooth submersions, then it implies that $\Phi_{k,\ell}$ and $\Phi_{k,\ell}^{-1}$ are both differentiable and thus $\Phi_{k,\ell}$ is a diffeomorphism. 
    
    Recall that $\pi_{k,\ell}$ is a smooth submersion. It remains to show that $\phi_{k,\ell}$ is a smooth submersion. Since $\phi_{k,\ell}$ is a surjective smooth map, it suffices to show that the differential $d_{X}(\phi_{k,\ell})$ at any $X\in\Pi^m_{\ell}\SSS^{k-1}$ is surjective. Here the differential map is given by
    \$
    d_{X}(\phi_{k,\ell}):\cT_{X}\Pi^m_\ell\SSS^{k-1}\to \cT_{\phi_{k,\ell}(X)}\Corr(m,\ell),\quad W\to W X^\top +X W^\top. 
    \$
    The space $\Pi_{\ell}^m\SSS^{k-1}$ can be rewritten as
    \#\label{equ:os}
    \Pi_\ell^m\SSS^{k-1}=\{YU^\top\mid Y\in\Pi_{\ell}^m\SSS^{\ell-1}, U\in\St(k,\ell)\}.
    \#
    The tangent space of $\Pi_{\ell}^m\SSS^{k-1}$ at each point  $X=YU^\top$ is given by
    \#\label{equ:tangent-os}
    \cT_{X}\Pi^m_\ell\SSS^{k-1}=\{TU^\top + YK^\top\mid\ & T\in\RR^{m\times \ell},\diag(YT^\top)=\zero,\notag\\
    &K\in\RR^{k\times \ell},U^\top K+K^\top U=\zero\}.
    \#
    Therefore, for $X=YU^\top$ in \eqref{equ:os} and $W=TU^\top+YK^\top$ in \eqref{equ:tangent-os}, we have 
    \$
    d_{X}(\phi_{k,\ell})(W)&=(TU^\top+YK^\top)(YU^\top)^\top+YU^\top (TU^\top+YK^\top)^\top\\
    &=TY^\top+YT^\top.
    \$
    By the characterization \eqref{equ:tangent_corr_concise} of the tangent space of $\Corr(m,\ell)$, we conclude that the differential $d_X(\phi_{k,\ell})$ is surjective. This completes the proof.
\end{proof}

\subsection{Proof of Proposition \ref{prop:counterexample}}\label{sec:b2}

\begin{proposition}
    Let $d^{\orb,k}$ be the distance in the orbit space $\Corr(m,[k])$. For $k_1<k_2\leq m$, we have 
    \$
    d^{\orb,k_2}(Z_1,Z_2)\leq d^{\orb,k_1}(Z_1,Z_2),\quad \forall Z_1,Z_2\in\Corr(m,[k_1]).
    \$
    Moreover, $d^{\orb,k_1}$ and $d^{\orb,k_2}$ are different, even when constrained to $\Corr(m,[k_1])$.
\end{proposition}

\begin{proof}
    Let $k_1<k_2\leq m$. First, we show that 
    \$
    d^{\orb,k_2}(Z_1,Z_2)\leq d^{\orb,k_1}(Z_1,Z_2),\quad \forall Z_1,Z_2\in\Corr(m,[k_1]).
    \$
    Let $X,Y\in\Pi^m\SSS^{k_1-1}$ and {set $\tilde X=[X\ \zero],\tilde Y=[Y\ \zero]\in\Pi^m\SSS^{k_2-1}$.} The distances $d^{\orb,k_1}$ and $d^{\orb,k_2}$ are given by
    \#
    (d^{\orb,k_1}(XX^\top,YY^\top))^2&=\inf_{O\in\cO(k_1)}\sum_{i=1}^m\arccos^2(\inner{(XO)_{i\cdot}}{Y_{i\cdot}}),\label{equ:d1}\\
    (d^{\orb,k_2}(XX^\top,YY^\top))^2&=\inf_{\tilde O\in\cO(k_2)}\sum_{i=1}^m\arccos^2(\langle{(\tilde X\tilde O)_{i\cdot}},{\tilde Y_{i\cdot}}\rangle)\notag\\
    &=\inf_{U\in\cA}\sum_{i=1}^m\arccos^2(\inner{(XU)_{i\cdot}}{Y_{i\cdot}}),\label{equ:d2}
    \#
    where $\cA$ is given by
    \$
    \cA=\{U\in\RR^{k_1\times k_1}\mid \exists \tilde O\in\cO(k_2) \textnormal{ such that }U=\tilde O_{[1:k_1],[1:k_1]}\}.
    \$
    Since $\cO(k_1)$ is a subset of $\cA$, we have
    \$
    d^{\orb,k_1}(XX^\top,YY^\top)\geq  d^{\orb,k_2}(XX^\top,YY^\top).
    \$
    This proves the first part of the proposition.
    
    Now we demonstrate that $d^{\orb,k_1}$ is different from $d^{\orb,k_2}$. Due to the monotone property we have proved, it suffices to consider the case where $k_2=k_1+1$. We use examples to prove the claim. Specifically, let
    \$
    X=\begin{pmatrix}
        -1 & \zero\\
        1 & \zero \\
        \zero & W
    \end{pmatrix},Y=\begin{pmatrix}
        1 & \zero \\
        1 & \zero \\
        \zero & W
    \end{pmatrix}\in\Pi^{m}\SSS^{k_1-1},
    \$
    where $W\in\Pi^{m-2}_{k_1-1}\SSS^{k_1-2}$ is of rank $k_1-1$. Also, let 
    \$
    \tilde X=(\zero, X),\tilde Y=(\zero,Y) \in\Pi^m\SSS^{k_1}.
    \$
    Let
    \$
    O^{**}=\begin{pmatrix}
        0 & 1 & \zero \\
        1 & 0 & \zero \\
        \zero & \zero & I_{k_1-1}
    \end{pmatrix}\in\cO(k_1+1).
    \$
    Then by definition, the distance $d^{\orb,k_1+1}(\tilde X,\tilde Y)$ is smaller than
    \$
    d^{\orb,k_1+1}(\tilde X\tilde X^\top,\tilde Y\tilde Y^\top)\leq d^{\Pi\SSS,k_1+1}(\tilde X O^{**},\tilde Y)=\pi/\sqrt{2}.
    \$
    In the meanwhile, since $\cO(k_1)$ is a compact set, we can express $d^{\orb,k_1}(\cdot,\cdot)$ as follows:
    \$
    d^{\orb,k_1}(XX^\top,YY^\top)=d^{\Pi \SSS,k_1}(XO^*,Y)
    \$ 
    for some $O^*\in\cO(k_1)$. There are two cases: 
    \vspace{2pt}
    {
    \renewcommand{\labelitemi}{\scalebox{0.4}{$\bullet$}}
    \begin{itemize}
        \item First, $O^*_{11}\neq0$, where $O^*_{11}$ is the $(1,1)$-th element of $O^*$. In this case, 
        \$
        d^{\Pi \SSS,k_1}(XO^*,Y)\geq \sqrt{\arccos^2(O^*_{11})+\arccos^2(-O^*_{11})}>\pi/\sqrt{2}.
        \$
         
        \item Second, $O^*_{11}= 0$. In this case, we denote $O^*_{2:,2:}$ the $(2:k_1,2:k_1)$-th submatrix of $O^*$. {Since $O^*$ is an orthogonal matrix and $O^*_{11}=0$, we conclude that  $O^*_{2:,2:}$ is not equal to  $I_{k_1-1}$.} As a result, we have
        \$
        d^{\Pi\SSS,k_1-1}(W O^*_{2:,2:},W)>0,
        \$
        where we use the condition that $W$ is full rank (i.e., rank $(k_1-1)$).
        Then, by definition, 
        \$
        d^{\Pi\SSS,k_1}(XO^*,Y)\geq \sqrt{\pi^2/2+(d^{\Pi\SSS,k_1-1}(W O^*_{2:,2:},W))^2}>\pi/\sqrt{2}.
        \$
    \end{itemize}
    }
    \noindent Combining these two cases, we obtain
    \$
    d^{\orb,k_1}(XX^\top,YY^\top)=d^{\Pi\SSS,k_1}(XO^*,Y)> \pi/\sqrt{2}\geq d^{\orb,k_1+1}(\tilde X\tilde X^\top,\tilde Y\tilde Y^\top).
    \$
    This concludes the proof, since $XX^\top=\tilde X\tilde X^\top$ and $YY^\top=\tilde Y\tilde Y^\top$.
\end{proof}





\bibliographystyle{siamplain}
\bibliography{references}
\end{document}

%% file: ex_shared.tex
\usepackage{lipsum}
\usepackage{amsfonts}
\usepackage{graphicx}
\usepackage{epstopdf}
\usepackage{algorithmic}
\ifpdf
  \DeclareGraphicsExtensions{.eps,.pdf,.png,.jpg}
\else
  \DeclareGraphicsExtensions{.eps}
\fi

\newsiamremark{remark}{Remark}
\newsiamremark{hypothesis}{Hypothesis}
\crefname{hypothesis}{Hypothesis}{Hypotheses}
\newsiamthm{claim}{Claim}

\headers{Quotient geometry of correlation matrices}{H. Chen}

\title{Quotient Geometry of Bounded or Fixed-Rank Correlation Matrices\thanks{Submitted to the editors TBA.
}}

\author{Hengchao Chen\thanks{Department of Statistical Sciences, University of Toronto
  (\email{hengchao.chen@mail.utoronto.ca}).}
}

\usepackage{amsopn}
\DeclareMathOperator{\diag}{diag}
